\documentclass[11 pt,fullpage]{article}
\usepackage{amsfonts,amssymb}
\usepackage{lipsum,amsmath,amsthm}
\usepackage{ dsfont }
\usepackage{color} 
\usepackage{hyperref}
\usepackage{mathtools}
\usepackage[right = 2.5cm, left=2.5cm, top = 2.5cm, bottom =2.5cm]{geometry}
\usepackage[utf8]{inputenc}
\usepackage[english]{babel}

\theoremstyle{plain}
\newtheorem{theorem}{Theorem}

\newtheorem{proposition}{Proposition}[section]
\newtheorem{lemma}[proposition]{Lemma}

\newtheorem{corollary}{Corollary}

\theoremstyle{definition}
\newtheorem{remark}{Remark}[section]

\newtheorem{innercustomthm}{Theorem}
\newenvironment{customthm}[1]
  {\renewcommand\theinnercustomthm{#1}\innercustomthm}
  {\endinnercustomthm}

\allowdisplaybreaks

\numberwithin{equation}{section}

\newcommand\R{{\mathbb R}}
\newcommand{\Z}{\mathbb{Z}}
\newcommand{\Q}{\mathbb{Q}}
\newcommand\Torus{{\mathbb T}}

\newcommand{\F}{\mathcal{F}}
\newcommand{\dt}{\partial_t}
\newcommand{\dx}{\partial_x}

\newcommand{\dz}{\partial_z}
\newcommand{\db}{\partial_\sigma}
\newcommand{\del}{\nabla}

\newcommand{\les}{\lesssim}
\newcommand{\ges}{\gtrsim}
\newcommand{\jap}[1]{\left\langle #1 \right\rangle}
\newcommand{\inp}[2]{\left\langle #1,#2 \right\rangle}
\newcommand{\lap}{\Delta}
\newcommand{\wsint}{\int_{\Torus \times \R \times \Torus}}
\newcommand{\HL}[2]{#1^\text{Hi}#2^\text{Lo}}

\newcommand{\n}[1]{#1_{\neq}}
\newcommand{\z}[1]{#1_0}
\newcommand{\djj}{\partial_j}
\newcommand{\di}{\partial_i}
\newcommand{\dij}{\partial_{ij}}

\newcommand{\dXY}{\partial_{XY}}
\newcommand{\dX}{\partial_X}
\newcommand{\dY}{\partial_Y}
\newcommand{\dZ}{\partial_Z}
\newcommand{\dB}{\partial_\sigma}
\newcommand{\til}[1]{\tilde{#1}}
\newcommand{\ind}{\mathds{1}}
\newcommand{\te}{\text}
\newcommand{\lam}{\lambda}
\newcommand{\DN}{\jap{\del}^N}

\newcommand{\pz}{P_{l \neq 0}}

\begin{document}
 
\title{On the Sobolev stability threshold of 3D Couette flow in a homogeneous magnetic field } 
\author{Kyle Liss\footnote{\textit{kliss@math.umd.edu}}}

\date{\today}
\maketitle

\begin{abstract} 
We study the stability of the Couette flow $(y,0,0)^T$ in the 3D incompressible magnetohydrodynamic (MHD) equations for a conducting fluid on $\Torus \times \R \times \Torus$ in the presence of a homogeneous magnetic field $\alpha(\sigma, 0, 1)$. We consider the inviscid, ideal conductor limit $\textbf{Re}^{-1}$, $\textbf{R}_m^{-1} \ll 1$ and prove that for strong and suitably oriented background fields the Couette flow is asymptotically stable to perturbations small in the Sobolev space $H^N$. More precisely, we show that if $\alpha$ and $N$ are sufficiently large, $\sigma \in \R \setminus \Q$ satisfies a generic Diophantine condition, and the initial perturbations $u_{\te{in}}$ and $b_{\te{in}}$ to the Couette flow and magnetic field, respectively, satisfy $\|(u_{\te{in}},b_{\te{in}})\|_{H^N} = \epsilon \ll \textbf{Re}^{-1}$, then the resulting solution to the 3D MHD equations is global in time and the perturbation $(u(t,x+yt,y,z),b(t,x+yt,y,z))$ remains $\mathcal{O}(\textbf{Re}^{-1})$ in $H^{N'}$ for some $N'(\sigma) < N$. Our proof establishes enhanced dissipation estimates describing the decay of the $x$-dependent modes on the timescale $t \sim \textbf{Re}^{1/3}$, as well as inviscid damping of the velocity and magnetic field that agrees with the optimal decay rate for the linearized system. In the Navier-Stokes case, high regularity control on the perturbation in a coordinate system adapted to the mixing of the Couette flow is known only under the stronger assumption $\epsilon \ll \textbf{Re}^{-3/2}$ \cite{BGM15I}. The improvement in the MHD setting is possible because the magnetic field induces time oscillations that partially suppress the lift-up effect, which is the primary transient growth mechanism for the Navier-Stokes equations linearized around the Couette flow.
\end{abstract}

\setcounter{tocdepth}{1}
{\small\tableofcontents}

\section{Introduction} \label{sec:intro}

\subsection{Problem statement and background} \label{intro:problemstatement}
In this paper, we consider the 3D incompressible MHD equations set on $\Torus \times \R \times \Torus$:
\begin{equation}{\label{eq:general}}
\begin{cases}\dt \tilde{u} + \tilde{u}\cdot \del \tilde{u} - \tilde{b} \cdot \del \tilde{b} = -\del \tilde{p} + \nu \lap \tilde{u}, \\ 
\dt \tilde{b} + \tilde{u}\cdot \del \tilde{b} - \tilde{b}\cdot \del \tilde{u} = \mu \lap \tilde{b}, \\ 
\del \cdot \tilde{u} = \del \cdot \tilde{b} = 0.
\end{cases}
\end{equation}
Here, $\nu = \textbf{Re}^{-1} > 0$ is the inverse Reynolds number, $\mu = \textbf{R}_{\text{m}}^{-1} > 0$ is the inverse magnetic Reynolds number, and $\Torus$ is the periodized interval $[0,1]$. The functions $\tilde{u}:  \R^+ \times \Torus \times \R \times \Torus \to \R^3$, $\tilde{b}:  \R^+ \times \Torus \times \R \times \Torus \to \R^3$, and $\tilde{p}: \R^+ \times \Torus \times \R \times \Torus \to \R$ denote the velocity, magnetic field, and pressure, respectively, and we write $(t,x,y,z) \in \R^+ \times \Torus \times \R \times \Torus$. \par 
Perhaps the simplest stationary solution to (\ref{eq:general}) with a nonzero velocity is the Couette flow $u_s = (y,0,0)^T$ in any homogeneous magnetic field $b_s = \alpha(\sigma, 0, 1)^T$ ($\alpha$, $\sigma \in \R$). Analyzing the stability of this solution in the inviscid, ideal conductor limit ($\nu$, $\mu \to 0$) serves as a model problem for understanding shear flow stability in magnetized plasmas, an area which has received considerable attention in the past \cite{Stuart1954,Chandrasekhar293,Hunt342,HughesTobias,Velik,chandrasekhar1981hydrodynamic}.
For an overview of MHD we refer to \cite{Davidson}, and for general texts on hydrodynamic stability theory see \cite{DR81,Yaglom}. To study the stability of $(u_s,b_s)$ we introduce the perturbations $u$ and $b$ defined by $\tilde{u} = u + u_s$ and $\tilde{b} = b + b_s$. They satisfy the system (denoting the component of a vector with a superscript)
\begin{equation}{\label{eq:origcoords}}
\begin{cases}\dt u + u\cdot \del u - b\cdot \del b + y\dx u  -\alpha \db b + \begin{pmatrix} u^2 \\ 0 \\ 0 \end{pmatrix} = -\del p^{NL} + 2\del \lap^{-1}\dx u^2 + \nu \lap u, \\ 
\dt b + u\cdot \del b - b\cdot \del u + y\dx b - \alpha \db u - \begin{pmatrix} b^2 \\ 0 \\ 0 \end{pmatrix} = \mu \lap b, \\ 
p^{\te{NL}} = (-\lap)^{-1}(\djj u^i \di u^j - \djj b^i \di b^j), \\ 
\del \cdot u = \del \cdot b = 0,\\ 
u(0) = u_{\text{in}}, \quad b(0) = b_{\text{in}},
\end{cases}
\end{equation}
where summation over repeated indices is implied and we have defined the directional derivative $\db = \sigma \dx + \dz$. \par 
A potential formulation of the nonlinear stability problem for (\ref{eq:origcoords}) is motivated by the phenomenon in 3D hydrodynamics known as \textit{subcritical transition}, which refers to when a linearly stable flow (see \cite{BGMreview,DR81} for precise definitions) is nevertheless experimentally unstable and transitions to turbulence at sufficiently high Reynolds number. Flow through a pipe, studied by Reynolds in his original experiments, provides a classic example. Indeed, laminar pipe flow becomes turbulent in experiments at sufficiently high Reynolds number, and yet numerical calculations suggest that the linearized system is spectrally stable (for any Reynolds number) \cite{DR81}. Distinct from this example is plane Poiseuille flow, which is linearly unstable for high enough Reynolds number, but typically transitions to turbulence in experiments well below the critical Reynolds number predicted by the linear theory \cite{CWP,Yaglom}. An idea dating back to Kelvin \cite{Kelvin87} to reconcile the linear stability with the experimental instability is that while a given flow might be nonlinearly stable for any fixed Reynolds number, its basin of attraction shrinks as $\nu \to 0$. The equilibrium is then unstable in practice at sufficiently high Reynolds number due to the inevitable presence of finite amplitude perturbations in experiments. This suggests that the natural problem is to quantify the largest perturbation possible, with respect to the Reynolds number, such that a given system does not transition to turbulence. \par 
For spectrally stable hydrodynamic shear flows (e.g., variations of Couette flow \cite{DR81}), transient growth mechanisms originating in the nonnormality of the linearized operator play an important role in the transition to turbulence \cite{Trefethen}. Similar effects can contribute to subcritical transition for magnetized shears \cite{MGCmhdshear}, and hence it is natural to extend Kelvin's idea to our MHD setting and formulate the stability problem for (\ref{eq:origcoords}) as follows \cite{BGMreview}:
\begin{equation}
  \tag{Q}\label{eq:question}
  \parbox{\dimexpr\linewidth-4em}{%
    \strut
    \textit{Given an initial norm $X_i$ and a final norm $X_f$ what are the smallest $\beta(X_i,X_f), \gamma(X_i,X_f) \ge 0$ such that if the initial perturbations $u_\text{in}$ and $b_\text{in}$ satisfy $$\mu^{-\beta}\|b_{in}\|_{X_i} + \nu^{-\gamma}\|u_{in}\|_{X_i} \ll 1$$ 
    then the solution is global in time, does not transition away from $(u_s,b_s)$, and converges back to $(u_s,b_s)$ as $t \to \infty$ in the sense that $$\|(u,b)\|_{L^{\infty} X_f} \ll 1, \quad \lim_{t \to \infty}\|(u(t),b(t))\|_{X_f} = 0?$$}
    \strut}
\end{equation}
Our goal in this paper is to contribute to the answer of (\ref{eq:question}) when $X_i$ and $X_f$ are Sobolev spaces adapted to the linear dynamics and in the special case where the ideal limit is taken with $\mu = \nu$ (we will henceforth write $\nu = \textbf{Re}^{-1} = \textbf{R}_m^{-1}$). Adopting the standard terminology, we refer to the number $\gamma \ge 0$ as the \textit{transition threshold}. It is not known a priori that the basin of attraction necessarily shrinks as a power law. Moreover, studies of the Couette flow in 2D \cite{BMV14,BVW16} and 3D \cite{BGM15I,BGM15II,BGM15III,WZ18} in the Navier-Stokes equations suggest that $\gamma$ might depend in a complicated way on the norms $X_i$ and $X_f$. 
\par

The transient growth mechanism most responsible for subcritical transition of 3D Couette flow in the Navier-Stokes equations is the lift-up effect. First noticed for more general shear flows in \cite{Ellingsen}, it predicts, in the linearized equations, linear in time growth of the velocity for $t \les \frac{1}{\nu}.$ Other growth mechanisms that play a role in the stability analysis include an algebraic growth in time of derivatives caused by the mixing, and an amplification of certain Fourier modes due to a transient unmixing of high frequency information to large scales. This latter effect was first noticed by Orr in \cite{Orr07} and is known as  the \textit{Orr mechanism}. We will need to contend with these same effects in our study of (Q) in the MHD setting. There is a crucial difference, however, in that the presence of a transverse magnetic field component partially suppresses the lift-up effect. This observation plays a fundamental role in the proof of our main result and is discussed in Sec. \ref{sec:linearlu}.
\subsection{Previous work} \label{sec:previous}

There is a substantial body of mathematical results on the analog of (\ref{eq:question}) for the Couette flow in the Navier-Stokes equations. The works most related to our present study are \cite{BGM15I} and \cite{WZ18}, which prove, in distinct senses, that $\gamma \le 3/2$ and $\gamma \le 1$, respectively, when the domain is $\Torus \times \R \times \Torus$ and $X_i$ and $X_f$ are Sobolev spaces. More specifically, the result in \cite{BGM15I} shows that if $\|u_{\text{in}}\|_{H^s} \ll \nu^{3/2}$ for any $s > 9/2$ then there holds 
$$ \|U\|_{L^\infty H^{s-2}} = \mathcal{O}(\nu^{1/2}), \quad \lim_{t \to \infty}\|\n{U}\|_{H^{s-2}} = 0, $$
where the subscript $\neq$ denotes the projection onto nonzero frequencies in $x$ (see Sec.~\ref{outline:notations}) and $U(t,x,y,z) = u(t,x + yt + t\psi(t,y,z),y - \psi(t,y,z),z)$ for a solution dependent function $\psi$ that remains $\mathcal{O}(\nu^{1/2})$ in $H^s$. The leading order effect of the coordinate transformation is to unwind by the mixing induced by the Couette flow, which amounts to modding out by the main component of the linear evolution. We thus refer to $U$, borrowing terminology from dispersive PDE, as the \textit{profile}. High regularity control on the profile gives quantitative information on the dynamics. For example, one can deduce that the mixing effect that characterizes the linear behavior persists as the leading order effect at the nonlinear level. Hence, the result in \cite{BGM15I} shows that for sufficiently regular initial data the solution looks essentially linear. On the other hand, the authors in \cite{WZ18} consider relatively low regularity ($H^2$ on the velocity variables) and prove that $\gamma \le 1$ when the derivatives are measured in the original coordinates. This result partially improves those in \cite{BGM15I} due to the weaker assumption on the initial data. It might be possible to extend the methods in \cite{WZ18} to obtain high regularity profile estimates for $\mathcal{O}(\nu)$ Sobolev data and thereby obtain a strict improvement on the results of \cite{BGM15I}, however, $\gamma \le 1$ in the sense of profile estimates on $\Torus \times \R \times \Torus$ is currently only known for infinite regularity perturbations lying in a Gevrey space \cite{BGM15II}. In fact, for Gevrey data it is possible to partially follow the lift-up instability, and hence even more precisely characterize the nonlinear dynamics \cite{BGM15II,BGM15III}. In this paper we take the approach of \cite{BGM15II,BGM15III,BGM15I} and prove profile estimates.  \par 
The known stability results in 2D are stronger because the lift-up effect is eliminated. In particular, for the Navier-Stokes equations on $\Torus \times \R$ it holds that $\gamma \le 1/2$ for Sobolev data \cite{BVW16} and $\gamma = 0$ for Gevrey data \cite{BMV14}. For the 2D Euler equations, asymptotic stability to Gevrey perturbations was proven on $\Torus \times \R$ in \cite{BMeuler} and, more recently, on the channel $\Torus \times [0,1]$ in \cite{Ionescu}. Moreover, it is known that the theorem in \cite{BMeuler} does not extend to Gevrey spaces weaker than those that it originally considered. This is a consequence of the recent work \cite{DM}, and leads to the conclusion that for Couette flow in the 2D Euler equations the dynamics of perturbations depend importantly on their regularity.


\subsection{Summary of main result}
We defer the complete statement of our main result until after we have discussed the linearization of (\ref{eq:origcoords}). However, it can be summarized as follows.
\begin{customthm}{1}[Summary] \label{thrm:sum}
Let $\mu = \nu \in (0,1]$ and $\sigma > 0$ be an irrational number that satisfies a generic Diophantine condition (see (\ref{eq:dioph1}) and (\ref{eq:dioph2}) for precise statements). Then, there exists $\til{c}(\sigma) \in \mathbb{N}$ such that for $\alpha$ and $N$ sufficiently large (depending only on $\sigma$) we have $\gamma(X_i = H^N, X_f= H^{N-\til{c}}) \le 1$ provided that $X_f$ measures derivatives on the linear profile $(U,B)$ defined by $(U(t,x,y,z),B(t,x,y,z)) = (u(t,x+yt,y,z),b(t,x+yt,y,z))$.
\end{customthm}

\begin{remark}
While taking $\mu = \nu$ is mathematically natural, it is usually not the case for real physical applications. Thus, it is of interest to consider $\mu \neq \nu$ and the more general double scaling problem posed in (Q). However, even for the simpler stationary solution $u_s = 0$, $b_s = (0,0,1)$, studying the 3D MHD equations with $\mu \neq \nu$ is known to create substantial mathematical difficulties (see \cite{WZ17} and the references therein), and our current proof makes heavy use of the $\mu = \nu$ structure.
\end{remark}

\subsection{Brief discussion of results and ideas of the proof} \label{sec:discussion}
Theorem~\ref{thrm:sum} shows that a sufficiently strong magnetic field with an appropriate irrational tilt has a stabilizing effect on the Couette flow. Indeed, our result establishes high regularity profile estimates for Sobolev space perturbations on $\Torus \times \R \times \Torus$ akin to those in \cite{BGM15I}, but under the weaker assumption of $\mathcal{O}(\nu)$ initial data. Moreover, while~(\ref{eq:origcoords}) is more complicated than the Navier-Stokes equations, our proof in many respects is much simpler than those in \cite{BGM15I,WZ18}. Most notably, our proof does not require a solution-dependent nonlinear change of coordinates. We are also able to treat the various estimates more generally since, without the lift-up effect, $u^1$ and $u^3$ behave essentially the same. \par

The fact that strong magnetic fields can have a stabilizing effect on a conducting fluid has been observed in the literature. For example, the works \cite{Chandrasekhar293,ChandrInv,Velik} show that a sufficiently large magnetic field can nullify linear instabilities in Taylor-Couette flow. Moreover, simulations in \cite{LiuChen} demonstrate that a background field parallel to a free shear layer tends to suppress the Kelvin-Helmholtz instability. 
For our present case of plane periodic Couette flow on $\Torus \times \R \times \Torus$, we find that the stabilization occurs because, as mentioned above, the magnetic field partially suppresses the lift-up effect. To understand this physically, note that the lift-up effect in the Navier-Stokes equations occurs as the fluid circulates in planes normal to the direction of the streamwise flow, which redistributes the mean streamwise velocity and can drastically alter the shear profile \cite{Ellingsen}. Now, when the fluid is electrically conducting and a sufficiently strong transverse magnetic field is present, the field lines provide a restoring force via the frozen-in-law that resists the rotation of fluid layers. Thus, instead of growth, in the MHD setting oscillations occur and are transmitted in the form of Alfv\'{e}n waves. See Sec.~\ref{sec:linearlu} for a detailed mathematical discussion. \par 

A key idea in the proof of Theorem 1 is that we can capture the oscillations induced by the magnetic field by using integration by parts in time via the identity 
\begin{equation} \label{eq:identity}
e^{i\omega(\mathbf{k}) t} = \frac{1}{i\omega(\mathbf{k})}\dt e^{i\omega(\mathbf{k}) t}.
\end{equation}
It turns out that for a Fourier mode with $\mathbf{k} = (k,\eta,l) \in \Z \times \R \times \Z$ the oscillation frequency behaves like $\omega(\mathbf{k}) \approx |\sigma k + l|$. One of the main difficulties in the proof is finding a way to utilize the oscillations for modes with $|\sigma k + l| \approx 0$; i.e., frequencies with wave vectors approximately perpendicular to the background magnetic field. This challenge underlies our idea to choose $\sigma$ irrational, as it provides a kind of non-resonance condition that ensures $|\sigma k + l| \neq 0$ for all $(k,l) \in \Z \times \Z$. For more details on the integration by parts in time and our strategy for absorbing losses incurred when the oscillation frequency tends to zero as $|k| \to \infty$ we refer to Secs.~\ref{sec:ibpt} and~\ref{outline:numbertheory} below. 
\par

Our proof also relies on the same stabilizing effects of the Couette flow utilized in the works \cite{BGM15I,BGM15II,BGM15III,BVW16,BMV14,WZ18} on the Navier-Stokes equations. In particular, the mixing induced by the Couette flow results in an improved (with respect to the heat equation) dissipation timescale for the $x$-dependent modes, which is referred to as \textit{enhanced dissipation}. A second stabilizing mechanism is the \textit{inviscid damping}, first discovered by Orr \cite{Orr07}, which causes decay on a timescale uniform in $\nu$. The linear analysis for the Navier-Stokes equations predicts an inviscid damping timescale of $\jap{t}^{-2}$, while in the MHD setting we have the difficulty that this is slowed to $\jap{t}^{-1}$ and is significantly harder to access (see Sec.~\ref{sec:linearinv}). To exploit the stabilizing properties at the nonlinear level we use the Fourier multiplier methods employed in \cite{BGM15I,BGM15II,BGM15III,BVW16,BMV14}. In this respect we follow most closely the ideas of \cite{BGM15I}. \par 

A few natural questions arise from Theorem~\ref{thrm:sum}. Firstly, it remains open whether or not the threshold estimate $\gamma \le 1$ holds in the case that $\sigma$ is either rational or violates (\ref{eq:dioph1}). It turns out however that for any $\sigma \in \R$ we can prove that $\gamma \le 4/3$ (see Corollary~\ref{cor:4/3}). Next, we expect that $\gamma < 1$ for Gevrey perturbations, which would be a significant result in that the lift-up effect suggests that no analogous result is possible for the Navier-Stokes equations. It is also of interest to consider domains with boundaries, and, as mentioned above, the physical case $\mu \neq \nu$ and the general double scaling limit suggested in (Q). 

\subsection{Notations and conventions} \label{outline:notations}
Given a vector $v = (v_j)_{j=1}^n$ we write $|v|$ to denote the $l^1$ norm. For $a \in \R$ we use the standard notation $\jap{a} = \sqrt{1 + a^2}.$ For two quantities $a$ and $b$ we write $a \les b$ to mean that there exists $C \ge 0$ such that $a \le Cb$. The constant $C$ may depend on the $N$, but is always independent of $\nu$, $\alpha$, $t_1$ and $t_2$ (both $t_1$ and $t_2$ are defined in Sec.~\ref{outline:bootstrap}). Sometimes we will write $a \les_{\beta} b$ if we want to emphasize that the implicit constant depends on some parameter $\beta$. All unlabeled integrals are assumed to be taken over $(x,y,z) \in \Torus \times \R \times \Torus$ and we use the shorthand notation $dV = dx\text{ }dy\text{ }dz$. \par 
Given a function $f:\Torus \times \R \times \Torus \to \R$ we define its Fourier transform $\hat{f}: \Z \times \R \times \Z \to \mathbb{C}$ by
$$\mathcal{F}(f) = \hat{f}(k,\eta, l) = \wsint{e^{-2\pi i(kx + \eta y + lz)}f(x,y,z)dV}.$$
The function $f$ is then recovered via the Fourier inversion formula 
$$f(x,y,z) = \sum_{k \in \Z} \int_{\eta \in \R}\sum_{l \in \Z} e^{2\pi i(kx + \eta y + lz)}\hat{f}(k,\eta,l)d\eta.$$
We denote the projection of $f$ onto the zero frequencies in $x$ by
$$
f_0 = \int_{\Torus}f(x,y,z)dx.
$$
Then, we write
$$\n{f} = f - f_0.$$
At times it will also be convenient to project onto the nonzero frequencies in $z$. For this, we use the alternate notation
$$\pz{f} = f - \int_{\Torus} f(x,y,z)dz.$$
For a general Fourier multiplier with symbol $m(k,\eta,l)$ we write $mf$ to denote $\mathcal{F}^{-1}(m(k,\eta,l)\hat{f})$, and we also adopt standard notations such as $|\del|$ for the multiplier with symbol $|k| + |\eta| + |l|$. Since $\sigma$ is fixed in the proof, we use, for any $a \in \R$, the shorthand notation
$$T_a^t = e^{a (\sigma \dx + \dz) t}$$
to denote the multiplier with symbol $e^{ia(\sigma k + l)t}$. We then write $\dt T_a^t$ to denote the Fourier multiplier with symbol $ia(\sigma k + l)e^{ia(\sigma k + l)t}$.
\par 
For $s \ge 0$ we define the Sobolev space $H^s$ using the norm 
$$\|f\|_{H^s} : = \|\jap{\del}^sf\|_{L^2},$$
where we adopt the shorthand $\jap{\del} = \jap{|\del|}$. For functions $f$ and $g$ we write the associated inner product as
$$\inp{f}{g}_{H^s} = \int \jap{\del}^s f \jap{\del}^s g \text{ }dV.$$  
For a function of space and time $f(t,x,y,z)$ defined on a time interval $(a,b)$ we define the Banach space $L^p(a,b;H^s)$ for $1 \le p \le \infty$ by the norm
$$\|f\|_{L^p(a,b;H^s)} = \|\|f\|_{H^s}\|_{L^p(a,b)}.$$
When the time interval is clear from context or mentioned explicitly elsewhere we use the shorthand notation $\|f\|_{L^p(a,b;H^s)} = \|f\|_{L^pH^s}$.

\section{Linear effects}
\label{intro:linearized}
Before proceeding to the specifics of our main theorem and its proof it is instructive to first discuss the linearization of (\ref{eq:origcoords}), which reads

\begin{equation} \label{eq:linear}
\begin{cases}
\dt u + y\dx u - \alpha \db b + \begin{pmatrix}u^2 \\ 0 \\ 0 \end{pmatrix} = 2\del \lap^{-1}\dx u^2 + \nu \lap u, \\ 
\dt b + y\dx b - \alpha \db u - \begin{pmatrix}b^2 \\ 0 \\ 0 \end{pmatrix} = \nu \lap b.
\end{cases}
\end{equation}
To study (\ref{eq:linear}) we make the natural coordinate transform that unwinds by the mixing of the Couette flow:
\begin{align*}
X &= x - yt, \\ 
Y & = y, \\
Z & = z.
\end{align*}
Denoting $B(t,X,Y,Z) = b(t,x,y,z)$ and $U(t,X,Y,Z) = u(t,x,y,z)$, (\ref{eq:linear}) then becomes
\begin{equation}
\begin{cases}
\dt U - \alpha \dB B + \begin{pmatrix} U^2 \\ 0 \\ 0 \end{pmatrix} = 2\del_L \lap_L^{-1} \dX U^2 + \nu \lap_L U, \\ 
\dt B - \alpha \dB U - \begin{pmatrix} B^2 \\ 0 \\ 0 \end{pmatrix} = \nu \lap_L B, \\
\end{cases}
\end{equation}
where $\del_L = (\dX^L,\dY^L,\dZ^L) = (\partial_X, \partial_Y - t \partial_X, \partial_Z)$, $\lap_L = \del_L \cdot \del_L$, and it is understood that $\db = \sigma \partial_X + \partial_Z$ when acting on functions in the new coordinates. In general, for a function $g(t,x,y,z)$ we will denote $G(t,X,Y,Z) = g(t,x,y,z)$.\par
Below we discuss the three linear effects that are crucial in the upcoming nonlinear analysis: the suppression of the lift-up effect due to the magnetic field, inviscid damping, and enhanced dissipation.

\subsection{Lift-up effect} \label{sec:linearlu}
We first recall the lift-up effect for the Navier-Stokes equations. For $\alpha = 0$ the velocity satisfies the system
\begin{equation} \label{eq:liftsys}
\dt u_0 + \begin{pmatrix} u^2_0 \\ 0 \\ 0 \end{pmatrix} = \nu \lap u_0,
\end{equation}
which we solve explicitly to obtain
\begin{align*}
u_0^1(t) &= e^{\nu t \lap}(u_0^1 (0) - t u_0^2 (0)), \\ 
u_0^2(t) &=  e^{\nu t \lap}u_0^2 (0), \\ 
u_0^3(t) &=  e^{\nu t \lap}u_0^3 (0).
\end{align*}
The lift-up effect refers to the linear in time growth of $u_0^1$ predicted by the formula above for $t \les \nu^{-1}$. In general, the best global estimate one can expect is
\begin{equation} \label{eq:lulinear}
\|u_0^1\|_{L^\infty H^s} + \nu^{1/2}\|\del u_0^1\|_{L^2 H^s} \les \nu^{-1}\|u_0(0)\|_{H^s}.
\end{equation}
\par 
Now we turn to the MHD case. The stabilizing effect of the magnetic field is easiest to see in the idealized equations, so we henceforth set $\nu = 0$ in this section. We introduce the Els\"{a}sser variables
\begin{equation} \label{eq:elsass1}
w^{\pm} = u \mp b,
\end{equation}
which represent waves that propagate along the direction of the background magnetic field. If we define the associated profiles (recall the definition $T_a^t = e^{at \db}$) 
\begin{equation} \label{eq:profiles1}
z^{\pm} = T_{\pm \alpha}^tw^{\pm},
\end{equation}
then we find that their projections onto the zero mode in $x$ solve the system
\begin{equation} \label{eq:lusystem}
\dt z_0^{\pm} + \begin{pmatrix} e^{\pm 2 \alpha t \dz} z_0^{\mp,2} \\ 0 \\ 0 \end{pmatrix} = 0,
\end{equation}
where we have noted that for any function $g$ we have $T_a^t g_0 = e^{a t \dz}g_0$.
While (\ref{eq:lusystem}) has a similar structure to (\ref{eq:liftsys}), the forcing term now exhibits oscillations that nullify the previously observed growth. By direct integration on the Fourier side we obtain the solution
\begin{align*}
\hat{z}^{\pm,1}(t,0,\eta,l) & = \hat{w}^{\pm,1}(0,0,\eta,l) - \frac{e^{\pm i\alpha lt}}{\alpha l} \sin(\alpha l t) \hat{w}^{\mp,2}(0,0,\eta,l), \\ 
\hat{z}^{\pm,2}(t,0,\eta,l) & = \hat{w}^{\pm,2}(0,0,\eta,l), \\ 
\hat{z}^{\pm,3}(t,0,\eta,l) & = \hat{w}^{\pm,3}(0,0,\eta,l),
\end{align*}
where we have noted that due to incompressibility we may assume that $l \neq 0$ in the time dependent piece of the formula for $z^{\pm,1}$. We then immediately obtain the estimate 
\begin{equation} \label{eq:nolulinear}
\|(u_0(t),b_0(t))\|_{H^s} \les \|(u_0(0),b_0(0))\|_{H^s} \quad \forall \text{ }s\ge 0,
\end{equation}
which is a tremendous gain over (\ref{eq:lulinear}).

\subsection{Diophantine approximation}
\label{outline:numbertheory}
To understand the inviscid damping we need facts about Diophantine approximation to quantify the possible losses incurred from integration by parts in time. For our purposes the following result, which is a consequence of Roth's theorem \cite{CasselNumber}, suffices. 
\begin{proposition} \label{prop:roth}
Let $t$ be an irrational algebraic number and fix any $r > 0$. Then, there exists a constant $C(t,r) > 0$ such that
\begin{equation} \label{eq:Roth}
    \left|t - \frac{p}{q}\right| > \frac{C}{|q|^{2+r}}
\end{equation}
for all rational $p/q$.
\end{proposition}
\noindent From Proposition~\ref{prop:roth}, it follows that if $\sigma \in \R^+ \setminus \Q$ is irrational and algebraic then for any $r > 0$ there holds, for all $s \ge 0$,
\begin{equation}\label{eq:dioph}
\|\db^{-1}\n{g}\|_{H^s} \les_{\sigma,r} \|\n{g}\|_{H^{s + 1 + r}}.
\end{equation}
For $n=n(\sigma)$ and $c = c(\sigma)$ as defined below in Theorem~\ref{thrm:main} the inequality above reads
\begin{equation} \label{eq:dioph2}
    \|\partial_{\sigma}^{-1} \n{g}\|_{H^s} \le \frac{1}{c}\|\n{g}\|_{H^{s + n}},
\end{equation}
which for the sake of consistency of notation is the form that we will employ in all that follows. As we will see in Sec.~\ref{sec:linearinv}, (\ref{eq:dioph2}) says that for the nonzero mode in $x$ terms we can integrate by parts in time at the cost of a losing $n$ derivatives. 

\begin{remark} \label{rem:number}
It is interesting to note that the set of real numbers for which there exists an $r > 0$ such that (\ref{eq:Roth}) fails to hold for any constant $C > 0$ has Lebesgue measure zero. Moreover, by Louiville's theorem on Diophantine approximation any irrational number that is algebraic of order two (i.e., is the root of a second degree polynomial with integer coefficients) satisfies (\ref{eq:Roth}) with $r = 0$ and a constant $C$ that is easy to quantify \cite{CasselNumber}. This result is in some sense sharp due to Dirichlet's theorem, which states that for any irrational number $t$ there exist infinitely many rational $p/q$ that satisfy $|t- p/q| < 1/q^2$.
\end{remark}

\subsection{Inviscid damping} \label{sec:linearinv}
Following the ideas of \cite{Kelvin87,BGM15I} we define the unknowns
\begin{equation} \label{eq:profile2}
f^\pm = \lap z^\pm. 
\end{equation}
A computation using (\ref{eq:linear}) shows that in the ideal case $F^{\pm,2}$ solves
\begin{equation} \label{eq:inviscidphysical}
\dt F^{\pm,2} + \dXY^L \lap_{L}^{-1} F^{\pm,2} = T_{\pm 2\alpha}^t\dXY^L \lap_{L}^{-1} F^{\mp,2}. 
\end{equation}
By looking on the Fourier side, we see that the term on the left-hand side contributes to growth for $t > \eta/k$. On the other hand, since the profiles themselves are not oscillating, we expect, in a similar spirit to what was found above in Sec.~\ref{sec:linearlu}, that the right-hand side of (\ref{eq:inviscidphysical}) should have a negligible effect over long times. Hence, for now we drop this term and discuss the validity of this crucial simplification below in Sec.~\ref{sec:ibpt}. Integration on the Fourier side then yields, for $k\neq 0$, the approximate solution
\begin{equation} \label{eq:linlapsol}
\hat{F}^{\pm,2}(t,k,\eta,l) \cong \sqrt{\frac{k^2 + (\eta - kt)^2 + l^2}{k^2 + \eta^2 + l^2}} \hat{F}^{\pm,2}(0,k,\eta,l),
\end{equation}
which predicts a linear in time growth rate (compare with (\ref{eq:model6}) below). From (\ref{eq:linlapsol}), the relation $|\hat{F}^{\pm,2}(t,k,\eta,l)| = (k^2 + (\eta - kt)^2 + l^2)|\hat{W}^{\pm,2}(t,k,\eta,l)|$, and $|k,\eta-kt,l|^{-1} \les \jap{t}^{-1}|k,\eta,l|$ we derive the inviscid damping estimate
\begin{equation} \label{eq:lineardamping}
\|(\n{U}^2, \n{B}^2)\|_{H^s} \les \jap{t}^{-1}\|(u^2_\te{in},b^2_\te{in})\|_{H^{s+2}}.
\end{equation}
The loss of regularity in (\ref{eq:lineardamping}) is physically meaningful and corresponds to the transient unmixing of information from small scales to large scales by the Couette flow. In particular, for $\eta k > 0$ with $|\eta| \gg |k|$ the velocity and magnetic field undergo a transient amplification on the time interval $[0,\eta/k]$: 
\begin{equation} \label{eq:orrmech}
\left|\frac{\hat{W}^{\pm,2}(t = \eta/k,k,\eta,l)}{\hat{W}^{\pm,2}(0,k,\eta,l)}\right| \sim \left|\frac{\eta}{k}\right| = t. 
\end{equation}
The decay (\ref{eq:lineardamping}) and the transient growth (\ref{eq:orrmech}) are together known as the \textit{Orr mechanism}, and the times $t = \eta/k$ are referred to as the \textit{Orr critical times}. Note that in the case of 3D Navier-Stokes the linearized system predicts $\jap{t}^{-2}$ inviscid damping of $\n{U}^2$. This decay timescale should not be possible for (\ref{eq:inviscidphysical}). In fact, numerical solutions to (\ref{eq:inviscidphysical}) were observed to grow linearly in time, in agreement with the approximation (\ref{eq:linlapsol}). 

\subsubsection{Quadratic growth of $F^{\pm,1}$ and $F^{\pm,3}$}
After some calculations we find that $F^{\pm,1}$ and $F^{\pm,3}$ satisfy
\begin{equation} \label{inviscid13}
\dt F^{\pm,j} + 2\dXY^L \lap_{L}^{-1} F^{\pm,j} + \ind_{j=1} T_{\pm 2\alpha}^t F^{\mp,2} = \djj^L \lap_{L}^{-1} \dX F^{\pm,2} + T_{\pm 2\alpha}^t\djj^L \lap_{L}^{-1} \dX F^{\mp,2}. 
\end{equation}
It follows from the factor of $2$ in the term $2\dXY^L \lap_{L}^{-1} F^{\pm,j}$ that $F^{\pm,j}$, $j = 1,3$, will in general grow quadratically in time. Thus, inverting $\lap_L$ does not yield a decay estimate for the first or third component of $(\n{U},\n{B})$.
\subsubsection{Integration by parts in time} \label{sec:ibpt}
Now we return to the approximation that the oscillating term is a lower order contribution to (\ref{eq:inviscidphysical}). In particular, we want to show that its inclusion does not spoil the linear growth estimate (\ref{eq:linlapsol}) that yields (\ref{eq:lineardamping}). The main idea is to integrate by parts in time and use that the time derivatives of the profiles gain a factor $\jap{t}^{-1}$ of time decay in comparison to the profiles themselves. 
To demonstrate the key points we consider a model equation that captures the same growth and oscillation timescales of (\ref{eq:inviscidphysical}) but removes the frequency dependence. Specifically, we consider
\begin{equation}\label{eq:model}
 \dt F^{\pm} - \frac{1}{t} F^{\pm} = \frac{1}{t} T_{\pm 2 \alpha}^tF^{\mp}
\end{equation}
for $t \ge 1$ and $k \neq 0$, where for simplicity in this section we drop the second superscript. Noting that the left-hand side of (\ref{eq:model}) can be rewritten as $t\dt(t^{-1}F^{\pm})$, we obtain the a priori estimate
\begin{equation} \label{eq:model2}
\frac{1}{2}\|t^{-1}F^{\pm}(t)\|^2_{L^2} = \frac{1}{2}\|F^\pm(1)\|_{L^2}^2 + \int_1^t \int s^{-3} F^{\pm} T_{\pm 2\alpha}^s F^{\mp}dVds.
\end{equation}
Using (\ref{eq:identity}) we integrate by parts in time and use Plancherel's theorem to rewrite the oscillating contribution as
\begin{align}
    \Big|\int_1^t & \int s^{-3} F^{\pm} T_{2\alpha}^s F^{\mp}dVds\Big| = \Big|\int_1^t \int s^{-3} F^{\pm} \frac{1}{2\alpha} (\partial_s T_{2\alpha}^s) \db^{-1} F^{\mp}dVds\Big| \label{eq:model3}\\ 
    & \le \sum_{k,l \in \Z, k\neq 0}\int_{\eta \in \R}\int_1^t \frac{1}{2\alpha |\sigma k + l|}\left[s^{-1}|\partial_s (s^{-1} \hat{F}^{\pm})||s^{-1}\hat{F}^{\mp}|\right]d\eta ds \label{eq:model4}\\
    & + \text{symmetric terms } + \text{boundary terms}\nonumber,
\end{align}
where the symmetric terms correspond to the time derivative landing on the other two factors in the brackets above. From (\ref{eq:model}), it follows that (\ref{eq:model4}) and each of the symmetric terms gain one power of time decay in comparison to the left-hand side of (\ref{eq:model3}). Observe however that this gain costs a $|\sigma k + l|^{-1}$ factor, which, even for $\sigma \in \R \setminus \Q$, blows up as the oscillation frequency degenerates for $|k| \to \infty$. As mentioned above, the idea behind the Diophantine condition in Theorem~\ref{thrm:sum} is that we can absorb such losses by paying regularity. In fact, using (\ref{eq:dioph2}) and Cauchy-Schwarz it follows that for any $\theta > 0$ there holds
\begin{equation} \label{eq:model5}
(\ref{eq:model4}) \les_{\sigma,\theta}\frac{1}{\alpha}\|s^{-1} F^{\pm}\|_{L^\infty(1,t;L^2)}\|s^{-2 + \theta} F^{\mp}\|_{L^\infty(1,t;H^n)}.
\end{equation}
By estimating the other terms above in an similar fashion, one can check using a continuity argument that
\begin{equation} \label{eq:model6}
\|F^{+}(t)\|_{L^2} + \|F^{-}(t)\|_{L^2} \les t(\|F^{+}(1)\|_{H^n} + \|F^{-}(1)\|_{H^n})
\end{equation}
provided that $\alpha$ is sufficiently large and 
\begin{equation} \label{eq:model7}
    \|F^{\pm}(t)\|_{H^n} \les t^{2-\theta}(\|F^{+}(1)\|_{H^n} + \|F^-(1)\|_{H^n}).
\end{equation}
That is, a linear in time growth estimate that loses derivatives holds for (\ref{eq:model}) provided we have an estimate in a sufficiently higher norm that, while losing no regularity, allows for greater time growth. In particular, this analysis suggests that if (\ref{eq:model7}) is true, then the inviscid damping estimate (\ref{eq:lineardamping}) should follow provided we pay $n+2$ derivatives on the right-hand side. We conclude that uniform in $\nu$ time decay should be possible only if we pay regularity that depends on the choice of $\sigma$, and moreover that any optimal proof should require balancing estimates using varying time weights at multiple regularity levels. This observation is at the core of our proof of Theorem 1. 

\begin{remark}
In practice, especially at the nonlinear level, in the highest norm it is only feasible to show that (\ref{eq:model7}) holds with $\theta = 0$. This however does not pose any issue since one can simply iterate the argument above twice and pay $2n$ derivatives to close an estimate. On the other hand, dropping the frequency dependence of the coefficients in (\ref{eq:inviscidphysical}) is a substantial simplification. For the details of the argument above carried out on the complete nonlinear equations, see the estimates in Sec.~\ref{sec:intlso}.
\end{remark}
\begin{remark}
Proving an estimate like (\ref{eq:nolulinear}) for the nonlinear equations will also require integration by parts in time, but there will be no regularity losses since $1 \les |\sigma k + l|$ for $k = 0$, $l\neq 0$. See Sec.~\ref{sec:luhi} for the calculation.
\end{remark}

\subsubsection{Enhanced dissipation}
The modified Laplace operator $\lap_L$ leads to improved dissipation timescales. To see this, consider the model equation 
\begin{equation}
    \dt g = \nu\lap_L g,
\end{equation}
which on the Fourier side has solution
\begin{equation}
    \hat{g}(t,k,\eta,l) = \hat{g}(0,t,k,\eta,l)e^{-\nu \int_{0}^{t} (k^2 + (\eta - ks)^2 +l^2) ds}.
\end{equation}
Since $\int_{0}^{t} (k^2 + (\eta - ks)^2 + l^2) ds \ge k^2 t^3/12$, we obtain the estimate
\begin{equation}
    \|\n{g}(t)\|_{H^s} \le e^{-\nu t^3/12}\|\n{g}(0)\|_{H^s} \quad \forall \text{ }s\ge 0. 
\end{equation}
Hence, the nonzero modes decay on the timescale $t \sim \nu^{-1/3}$, which for $\nu \ll 1$ is an improvement on the $\nu^{-1}$ dissipation timescale of the usual heat equation.

\section{Statement of main results} \label{intro:result}
We are now ready to state our main result.

\begin{theorem} \label{thrm:main}
Let $\mu = \nu \in (0,1]$ and suppose that $\sigma \in \R^+ \setminus \mathbb{Q}$ is such that 
\begin{equation} \label{eq:dioph1}
    \inf_{p,q \in \Z}|q|^n\left|q\sigma - p\right| = c > 0
\end{equation} 
for some $n \ge 1$. Then, there exist universal constants $\delta$, $c_1 > 0$ such that for any $N \ge 11 + 3n$ there is a constant $c_0(N) > 0$ such that if $\alpha > c_1/c$ and
$$\|(u_\text{in},b_\text{in})\|_{H^{N+2}} = \epsilon \le c_0 \nu,$$
then the solution to (\ref{eq:origcoords}) is global in time and, denoting $N' = N - 4 - 2n$ and $N'' = N - 9 - 3n$, the profiles $U(t,X,Y,Z) = u(t,X+Yt,Y,Z)$ and $B(t,X,Y,Z) = b(t,X+Yt,Y,Z)$ satisfy the global estimates
\begin{subequations} \label{eq:thrm}
\begin{align}
\|e^{\delta \nu^{1/3}t} \lap_{X,Z}(\n{U}^2,\n{B}^2)\|_{L^\infty H^N} + \nu^{1/6}\|\lap_{X,Z}(\n{U}^2,\n{B}^2)\|_{L^2H^N} &\les \epsilon, \label{eq:thrmnon2}\\ 
\|(\n{U}^2,\n{B}^2)\|_{L^2H^{N'}} + \|\jap{t}\del_{X,Z}(\n{U}^2,\n{B}^2)\|_{L^\infty H^{N'-1}} & \les \epsilon, \label{eq:thrminviscid}\\
(j \in \{1,3\})\quad \|e^{\delta \nu^{1/3}t} \lap_{X,Z}(\n{U}^j,\n{B}^j)\|_{L^\infty H^{N''}} + \nu^{1/6}\|\lap_{X,Z}(\n{U}^j,\n{B}^j)\|_{L^2H^{N''}} &\les \epsilon, \label{eq:thrmnon1}\\ 
\|(\z{u},\z{b})\|_{L^\infty H^N} + \nu^{1/2}\|\del(\z{u},\z{b})\|_{L^2H^N} & \les \epsilon \label{eq:thrmzero},
\end{align}
\end{subequations}
where the implicit constants are independent of $\nu$, $N$, $n$, and $c$. 
\end{theorem}

\begin{remark}
The enhanced dissipation of the nonzero modes is described by the $e^{\delta \nu^{1/3}t}$ factors and the $\nu^{-1/6}$ scaling of the $L^2$ in time estimates in (\ref{eq:thrmnon2}) and (\ref{eq:thrmnon1}). Indeed, for $\nu \ll 1$ the $\nu^{-1/6}$ scaling is an improvement on the $\nu^{-1/2}$ scaling that holds for the heat equation. The inviscid damping is captured by the uniform in $\nu$ bounds in (\ref{eq:thrminviscid}). Notice in particular that the $\jap{t}^{-1}$ decay matches the optimal estimates predicted by the linear theory. The estimate (\ref{eq:thrmzero}) describes the suppression of the lift-up effect.
\end{remark}

\begin{remark}
The discussion in Remark~\ref{rem:number} implies that $n = 1$ is the minimal number satisfying (\ref{eq:dioph1}), and that for almost every $\sigma \in \R$ we may take $n = 1+r$ for any $r > 0$. Clearly then $n < 2$ is generic, however, the specific value of $n$ does not affect the structure of our proof, and so we take $n$ to be arbitrary to account for possibly exceptional circumstances.
\end{remark}

\begin{remark}
Given the result in \cite{WZ18}, it is reasonable to ask if some analog of Theorem~\ref{thrm:main} holds in low regularity if the initial data is taken as large as $\mathcal{O}(\nu^{\gamma})$ for some $\gamma < 1$ and the derivatives are measured in the original coordinates. From Sec.~\ref{sec:ibpt} we expect this to be a difficult problem since utilizing the inviscid damping, which should be key in any optimal proof, costs regularity beyond the usual Orr mechanism. It seems that studying the MHD stability problem in high regularity is most natural.
\end{remark}

In the case that $\sigma$ is arbitrary (possibly rational), the methods employed in the proof of Theorem~\ref{thrm:main} yield the following corollary. 

\begin{corollary} \label{cor:4/3}
Let $\mu = \nu \in (0,1]$, $\sigma \in \R$, and $\alpha_0$ be a sufficiently large universal constant. Then, for $\alpha > \alpha_0$ and any $N > 3/2$ we have $\gamma(X_i = H^{N+2},X_f = X^N) \le 4/3$.
\end{corollary}
\noindent Notice that $\gamma \le 4/3$ is still an improvement on the threshold estimate of $\gamma \le 3/2$ in \cite{BGM15I}. The gain is possible because even with rational $\sigma$ the presence of the magnetic field allows us to eliminate the lift-up effect in the zero mode. The gap between the results in Corollary~\ref{cor:4/3} and Theorem~\ref{thrm:main} arises because we lose the inviscid damping when $\sigma$ does not satisfy (\ref{eq:dioph1}). The proof then does not require a calculation analogous to that in Sec.~\ref{sec:intlso}. In fact, it only requires integration by parts in time in the zero mode lift-up term, which does not cause a loss of derivatives. We thus only need to perform estimates at a single regularity level, and hence the proof of Corollary~\ref{cor:4/3} is much simpler than that of Theorem~\ref{thrm:main}.

\section{Preliminaries and outline of the proof} \label{sec:outline}

\subsection{Frequency decompositions} \label{sec:freqdecomp}   
Since we perform estimates at various regularity levels, Fourier space decompositions play an important role in the proof. For our purposes it suffices to define the sharp cutoff function $\chi: \R^6 \to \R$ by 
$$\chi(\xi \in \R^3,\xi'\in \R^3) = \begin{cases} 1 & \text{if }|\xi - \xi'| \le 2|\xi'|, \\ 
0 & \text{otherwise}. \end{cases}$$
We then define the paraproduct decomposition
\begin{align*}
fg &= \F^{-1} \sum_{k',l' \in \Z}\int_{\eta' \in \R} \hat{f}(k',\eta',l') \hat{g}(k-k',\eta - \eta',l-l') \chi(k,\eta,l,k',\eta',l')d\eta' \\
& \quad + \F^{-1} \sum_{k',l' \in \Z}\int_{\eta' \in \R} \hat{f}(k',\eta',l') \hat{g}(k-k',\eta - \eta',l-l') (1-\chi(k,\eta,l,k',\eta',l'))d\eta' \\ 
& \quad : = 
f^{\text{Hi}}g^{\text{Lo}} + f^{\text{Lo}}g^{\text{Hi}}.
\end{align*}
From Plancherel's theorem, the triangle inequality, Young's inequality, and Sobolev embedding we have, for any $s > 0$ and $\kappa > 3/2$, 
\begin{equation} \label{eq:productrule}
\|\HL{f}{g}\|_{H^s} \les_{\kappa} \|f\|_{H^s}\|g\|_{H^{\kappa}}.
\end{equation}

\subsection{Reformulation of the equations} \label{outline:reformulation}

\subsubsection{New dependent variables}
We work in the coordinate system defined in Sec.~\ref{intro:linearized} and primarily on the unknowns $F^{\pm,i}$. Recall the definitions (\ref{eq:elsass1}), (\ref{eq:profiles1}), (\ref{eq:profile2}), the shorthand $T_a^t = e^{at\dB}$, and our convention to use capital letters to denote unknowns in the new coordinates. The unknowns $F^{\pm}$ satisfy \begin{align} \label{eq:reform1}
\dt F^{\pm,1} & + T_{\pm 2\alpha}^t Z^{\mp}\cdot \del_L F^{\pm,1} + T_{\pm 2\alpha}^t F^{\mp}\cdot \del_L Z^{\pm,1} + 2T_{\pm 2\alpha}^t \di^L Z^{\mp,j} \dij^L Z^{\pm,1}\nonumber + 2\dXY^L \lap_{L}^{-1} F^{\pm,1} \\ & + T_{\pm 2\alpha}^tF^{\mp,2} 
- \partial_{XX} \lap_{L}^{-1}(F^{\pm,2} + T_{\pm 2\alpha}^tF^{\mp,2}) = \dX(T_{\pm 2\alpha}^t \djj^L Z^{\mp,i}\di^L Z^{\pm,j}) + \nu \lap_L F^{\pm,1},
\end{align}
\begin{align} \label{eq:reform2}
\dt F^{\pm,2} & + T_{\pm 2\alpha}^t Z^{\mp}\cdot \del_L F^{\pm,2} + T_{\pm 2\alpha}^t F^{\mp}\cdot \del_L Z^{\pm,2} + 2T_{\pm 2\alpha}^t \di^L Z^{\mp,j} \dij^L Z^{\pm,2} \nonumber\\ 
& + \dXY^L \lap_{L}^{-1} F^{\pm,2} - T_{\pm 2\alpha}^t\dXY^L \lap_L^{-1}F^{\mp,2} = \dY^L(T_{\pm 2\alpha}^t \djj^L Z^{\mp,i}\di^L Z^{\pm,j}) + \nu \lap_L F^{\pm,2},
\end{align}
\begin{align} \label{eq:reform3}
\dt F^{\pm,3} & + T_{\pm 2\alpha}^t Z^{\mp}\cdot \del_L F^{\pm,3} + T_{\pm 2\alpha}^t F^{\mp}\cdot \del_L Z^{\pm,3} + 2T_{\pm 2\alpha}^t \di^L Z^{\mp,j} \dij^L Z^{\pm,3} + 2\dXY^L \lap_{L}^{-1} F^{\pm,3} \nonumber\\
& - \partial_{XZ} \lap_{L}^{-1}(F^{\pm,2} + T_{\pm 2\alpha}^tF^{\mp,2}) = \dZ(T_{\pm 2\alpha}^t\djj^L Z^{\mp,i}\di^L Z^{\pm,j}) + \nu \lap_L F^{\pm,3},
\end{align}
where summation over repeated indices is implied,  $i,j \in \{1,2,3\}$ corresponds to $\{X,Y,Z\}$ in the derivative operators, and we have written $T_{\pm 2\alpha}^t f g$ to mean $(T_{\pm 2\alpha}^t f ) g$ in the nonlinear terms. At times we will also work on the unknowns $Q = \lap_L U$ and $H = \lap_L B$, and in particular the second components. These satisfy
\begin{align}\label{eq:Q2newcoord}
\dt Q^2 & + Q \cdot \del_L U^2 + U\cdot \del_L Q^2 - H \cdot \del_L B^2 - B\cdot \del_L H^2 + 2\di^L U^j \dij^L U^2 \nonumber \\
& - 2\di^L B^j \dij^L B^2 - \alpha \db H^2  = \nu \lap_L Q^2 + \partial^L_Y (\djj^L U^i \di^L U^j - \djj^L B^i \di^L B^j)
\end{align}
and 
\begin{align}\label{eq:H2newcoord}
\dt H^2 & + Q \cdot \del_L B^2 + U\cdot \del_L H^2 - H \cdot \del_L U^2 - B\cdot \del_L Q^2 \nonumber \\
& + 2\di^L U^j \dij^L B^2 - 2\di^L B^j \dij^L U^2 + 2\dXY^L \lap_L^{-1} H^2 - \alpha \db Q^2  = \nu \lap_L H^2. 
\end{align}
Lastly, for certain estimates we work directly on $Z_0^\pm$, which solves
\begin{equation} \label{eq:reform4}
\begin{cases}
\dt Z_0^{\pm,1} + (T_{\pm 2\alpha}^t Z^{\mp}\cdot \del_L Z^{\pm,1})_0 + T_{\pm 2\alpha}^t Z_0^{\mp,2} = \nu \lap Z^{\pm,1}_0 \\ 
\dt Z_0^{\pm,2} + (T_{\pm 2\alpha}^t Z^{\mp}\cdot \del_L Z^{\pm,2})_0 = \nu \lap Z_0^{\pm,2} + \dY\lap^{-1}\dij (T_{\pm 2\alpha}^tZ^{\mp,i} Z^{\pm,j})_0 \\ 
\dt Z_0^{\pm,3} + (T_{\pm 2\alpha}^t Z^{\mp}\cdot \del_L Z^{\pm,3})_0 = \nu \lap Z_0^{\pm,3} + \dZ\lap^{-1}\dij (T_{\pm 2\alpha}^tZ^{\mp,i} Z^{\pm,j})_0.
\end{cases}
\end{equation}
\begin{remark} \label{remark:spacetimeres}
Observe the remarkable structure in (\ref{eq:reform1}) -- (\ref{eq:reform3}) and~(\ref{eq:reform4}) that the ``$+$'' variables never interact nonlinearly with the ``$-$'' variables. Physically speaking, all nonlinear interactions are between wavepackets transported in opposite directions along the magnetic field lines. On $\R^3$, this amounts to a dispersive effect whereby the waves themselves are not decaying (at least in the ideal case), but nevertheless the nonlinear terms decay as the interacting wavepackets separate in space \cite{He2017}. In the language of the spacetime resonance method for nonlinear wave equations (see, e.g., \cite{GermainSpacetime,GermainKlein,GNEulMax}), this structure means that the nonlinearity is \textit{space non-resonant} uniformly in frequency on $\R^3$. For our periodic setting, the effect of the relative transport is to provide time oscillations in all nonlinear interactions where the function containing $T_{\pm 2\alpha}^t$ has a nonzero $X$ or $Z$ frequency. For such interactions it is possible to integrate by parts in time, however, we do not know how to use this structure to obtain $\gamma < 1$ because the regularity losses discussed in Sec.~\ref{sec:ibpt} limit the possible gain in the high norm estimates.
\end{remark}

\subsubsection{Shorthands}
It will be useful to define some shorthands for the various terms above. For concreteness we will only discuss the terms in the form they appear in (\ref{eq:reform1}) -- (\ref{eq:reform3}) for the ``$+$'' variables. For the linear terms in the $F^{+,2}$ equation:
\begin{align*}
& \te{LS} = -\dXY^L \lap_{L}^{-1}F^{+,2} \hspace{0.98in} \te{(``linear stretch'')}, \\ 
& \te{OLS} = T_{2\alpha}^t\dXY^L \lap_{L}^{-1}F^{-,2} \hspace{0.75in} \te{(``oscillating linear stretch'')}.
\end{align*}
For the the linear terms in the $F^{+,\beta}$ equation, $\beta \in \{1,3\}$:
\begin{align*}
& \te{LU} = -T_{2\alpha}^t F^{-,2} \hspace{0.96in} (\te{``lift-up''}),\\ 
& \te{LS} = -2\dXY^L \lap_{L}^{-1}F^{+,\beta} \hspace{0.56in} (\te{``linear stretch''}),\\ 
& \te{LP1} = \partial_{X\beta}^L\lap_{L}^{-1}F^{+,2} \hspace{0.70in} (\te{``linear pressure''}),\\ 
& \te{LP2} = T_{2\alpha}^t \partial_{X\beta}^L\lap_{L}^{-1}F^{-,2} \hspace{0.48in} (\te{``linear pressure''}).
\end{align*}
Now we turn to the nonlinear terms for $F^{+,\beta}$, $\beta \in \{1,2,3\}$. In what follows, $i,j \in \{1,2,3\}$ and $s_1$, $s_2$ can be 0 or $\neq$. We denote the four types of terms by
\begin{align*}
& \te{NLT}(j,s_1,s_2) = -T_{2\alpha}^t Z_{s_1}^{-,j} \djj^L F^{+,\beta}_{s_2} \hspace{0.94in} (\te{``nonlinear transport''}),  \\ 
& \te{NLS1}(j,s_1,s_2) = -T_{2\alpha}^t F_{s_1}^{-,j} \djj^L Z^{+,\beta}_{s_2} \hspace{0.88in} (\te{``nonlinear stretch''}),\\ 
& \te{NLS2}(i,j,s_1,s_2) = -T_{2\alpha}^t\di Z^{-,j}_{s_1} \dij^L Z^{+,\beta}_{s_2} \hspace{0.64in} (\te{``nonlinear stretch''}),\\ 
& \te{NLP}(i,j,s_1,s_2) = \partial_{\beta}(T_{2\alpha}^t \djj^L Z_{s_1}^{-,i} \di^L Z_{s_2}^{+,j}) \hspace{0.53in} (\te{``nonlinear pressure''}).
\end{align*}
The generalization of these shorthands to the other equations is mostly clear, except for perhaps in the equations for $Q^2$ and $H^2$ since they have additional nonlinear terms. In this case we simply denote indifferently, for example, $U \cdot \del_L H^2$ and $B \cdot \del_L Q^2$ as nonlinear transport terms in the equation for $H^2$. This will not cause any confusion in the proof. We use superscripts $\te{HL}$ and $\te{LH}$ to denote the two pieces of a term corresponding to the paraproduct decomposition defined in Sec.~\ref{sec:freqdecomp}. We will also abuse notation slightly and use the same shorthands above to denote a term's contribution to an energy estimate. For example, in an $H^s$ energy estimate for $F^{+,\beta}$ we write one of the contributions from the nonlinear transport term as
$$\te{NLT}^{\te{HL}}(j,s_1,s_2) = -\int_{t_1}^{t_2} \int \jap{\del}^s F^{+,\beta} \jap{\del}^s(T_{2\alpha}^t Z_{s_1}^{-,j})^\te{Hi} (\djj F^{+,\beta}_{s_2})^\te{Lo}dVdt.$$
When we do not indicate $s_1$, $s_2$, or $j$ in the nonlinear shorthands we simply mean the term without any restrictions on the indices or frequency interactions. For example, we write $\te{NLT}(j) = -T_{2\alpha}^t Z^{-,j}\djj^L F^{+,\beta}$ and $\te{NLT} = -T_{2\alpha}^t Z^{-} \cdot \del_L F^{+,\beta}$.

\subsection{Fourier multiplier norm} \label{outline:Fouriernorm}
Inspired by the previous works \cite{BGM15I,BVW16,BGM15II,BGM15III}, our proof is based on energy estimates using weighted norms defined through Fourier multipliers. The multipliers that we employ have all, up to small modifications, been previously used in \cite{BGM15I}.

\subsubsection{Quadratic growth multipliers $m$ and $\til{m}$} \label{Fouriernorm:stretch}
The first class of multipliers we use are concerned with the natural quadratic in time growth that $F^{\pm,1}$ and $F^{\pm,3}$ experience, as well as the linear growth of $F^{\pm,2}$. Consider the model scalar equation 
\begin{equation} \label{eq:stretch}
\dt g + 2\dXY^L \lap_{L}^{-1}g = \nu \lap_L g.
\end{equation}
On the Fourier side this equation becomes 
\begin{equation} \label{eq:stretchfourier}
\dt \hat{g} + \frac{2 k (\eta - kt)}{k^2 + (\eta - kt)^2 + l^2}\hat{g} = -\nu(k^2 + (\eta - kt)^2 + l^2) \hat{g}.
\end{equation}
For $k \neq 0$ the term $\frac{2k(\eta - kt)}{k^2 + (\eta - kt)^2 + l^2} \hat{g}$ contributes to growth in $\hat{g}$ for $t \ge \eta / k$. On the other hand, for $k \neq 0$ the term on the right-hand side yields enhanced dissipation, which will overcome the growth for  $|t - \eta/k|$ sufficiently large with respect to some inverse power of $\nu$. In fact, one can check that for $k \neq 0$ there holds 
$$ \left|\frac{2k(\eta - kt)}{k^2 + (\eta - kt)^2 + l^2}\right| \le \frac{\nu}{32} (k^2 + (\eta - kt) + l^2) $$
whenever $|t- \eta/k| \ge 4 \nu^{-1/3}$. This motivates defining the multiplier $m$ by $m(0,k,\eta,l) = 1$ and the ODE
\begin{equation} \label{eq:m}
\frac{\dot{m}}{m} 
=
\begin{cases}
\frac{2k(\eta - kt)}{k^2 + (\eta - kt)^2 + l^2} & \te{ if }0\le t - \eta/k \le 4\nu^{-1/3}, \\ 
0 & \te{ otherwise}.
\end{cases}
\end{equation}
For certain unknowns it will also be useful to use a norm that weakens for each frequency indefinitely after the critical time. We thus define the modified multiplier $\til{m}$ by 
\begin{equation} \label{eq:mtil}
\frac{\dot{\til{m}}}{\til{m}} 
=
\begin{cases}
\frac{2k(\eta - kt)}{k^2 + (\eta - kt)^2 + l^2} & \te{ if }t \ge \eta/k, \\ 
0 & \te{ otherwise}.
\end{cases}
\end{equation}
From (\ref{eq:m}) and (\ref{eq:mtil}) we find that $m$ and $\til{m}$ are given by the exact formulas 
\begin{itemize}
\item $k = 0$: $m(t,0,\eta,l) = \til{m}(t,0,\eta,l) = 1$;
\item $k \neq 0$, $\eta k < 0$ and $|\eta| \ge 4\nu^{-1/3}|k|$:
$$m(t,k,\eta,l) = 1, \quad \til{m}(t,k,\eta,l) = \frac{k^2 + \eta^2 + l^2}{k^2 + (\eta - kt)^2 + l^2};$$
\item $k \neq 0$, $\eta k < 0$ and $|\eta| \le 4\nu^{-1/3}|k|$:
\begin{align*}
    & m(t,k,\eta,l) = \begin{cases} \frac{k^2 + \eta^2 + l^2}{k^2 + (\eta - kt)^2 + l^2} & \te{ if } t \in [0, \eta/k + 4\nu^{-1/3}), \\ \frac{k^2 + \eta^2 + l^2}{k^2 + (4k\nu^{-1/3})^2 + l^2} & \te{ otherwise }, \end{cases} \\
    & \til{m}(t,k,\eta,l) = \frac{k^2 + \eta^2 + l^2}{k^2 + (\eta - kt)^2 + l^2};
    \end{align*}
\item $k \neq 0$ and $\eta k > 0$: 
$$m(t,k,\eta,l) = \begin{cases}
1 & \te{ if }t \le \eta/k, \\ 
\frac{k^2 + l^2}{k^2 + (\eta - kt)^2 + l^2} & \te{ if }t\in (\eta/k,\eta/k + 4\nu^{-1/3}),  \\ 
\frac{k^2 + l^2}{k^2 + (4k\nu^{-1/3})^2 + l^2} & \te{ if } t \ge \eta/k + 4\nu^{-1/3};
\end{cases} 
$$
$$
\hspace{-2.9cm} \til{m}(t,k,\eta,l) = \begin{cases}
1 & \te{ if }t \le \eta/k, \\ 
\frac{k^2 + l^2}{k^2 + (\eta - kt)^2 + l^2} & \te{ if }t > \eta/k.
\end{cases}
$$
\end{itemize}
The natural multiplier to use in the norm for $F^{\pm,2}$, which is expected to grow linearly in time, is $m^{1/2}$. While it can be obtained from the formulas above, it is useful to know that it satisfies 
\begin{equation} \label{eq:mhalf}
\frac{\dot{m}^{1/2}}{m^{1/2}} = 
\begin{cases}
\frac{k(\eta - kt)}{k^2 + (\eta - kt)^2 + l^2} & \te{ if } 0 \le t - \eta/k \le 4\nu^{-1/3}, \\ 
0 & \te{ otherwise}.
\end{cases}
\end{equation}
The fundamental properties of $m$ and $\til{m}$ are summarized in the following lemma.

\begin{lemma}\label{lemma:stretch}
The multipliers $m$ and $\til{m}$ satisfy
\begin{subequations}
\begin{align}
& \til{m}(t,k,\eta,l) \le m(t,k,\eta,l) \le 1, \label{mle1}\\ 
& k^2 + l^2 \les (k^2 + (\eta - kt)^2 + l^2) \til{m}, \label{mlapl}\\ 
& \nu^{2/3} \les m(t,k,\eta,l), \label{mtrivcom}\\
& \frac{1}{\til{m}} + \frac{1}{m} \les \jap{t}^2, \label{mtrivcomt}\\ 
& \frac{\til{m}(t,k,\eta,l)}{\til{m}(t,k,\eta',l')} + \frac{m(t,k,\eta,l)}{m(t,k,\eta',l')} \les \jap{\eta - \eta'}^2 + \jap{l - l'}^2, \label{mcom}\\ 
& \til{m}(t,k,\eta,l) \les \frac{|k,\eta,l|^4}{\jap{t}^2}. \label{mtildecay} 
\end{align}
\end{subequations}
\end{lemma}
\noindent Except for (\ref{mcom}), the proof of Lemma~\ref{lemma:stretch} is essentially immediate from the exact formulas above. Inequality (\ref{mcom}) was proven for $m$ in \cite{BGM15I} and the proof for $\til{m}$ does not require any notable variations. Thus, we omit it for the sake of brevity. In the proof of Theorem~\ref{thrm:main} we will use (\ref{mle1}) -- (\ref{mtrivcomt}) so frequently that we will usually do so without any remark.

\subsubsection{Ghost multiplier $M$} \label{Fouriernorm:ghost}
We also introduce three additional multipliers $M_1$, $M_2$, and $M_3$. These multipliers are defined by $M_j(0,k,\eta,l) = 1$, $M_j(t,0,\eta,l) = 1$, and for $k \neq 0$ the differential equations
\begin{subequations}
\begin{align}
- \frac{\dot{M_1}}{M_1} &= \frac{k^2}{k^2 + l^2 + (\eta - kt)^2}, \label{eq:ghost1}\\
- \frac{\dot{M_2}}{M_2} &= \frac{\jap{kl}}{k^2 + l^2 + (\eta - kt)^2}, \label{eq:ghost2} \\ 
- \frac{\dot{M_3}}{M_3} &= 
\frac{\nu^{1/3} k^2}{k^2 + l^2 + \nu^{2/3}(\eta - kt)^2} \label{eq:ghostenh}.
\end{align}
\end{subequations}
We then define $M = M_1M_2M_3$ and observe that it satisfies 
$$-\frac{\dot{M}}{M} \ge -\frac{\dot{M}_j}{M_j}$$
for each $j \in \{1,2,3\}$. It follows readily by direct integration that there exists a universal constant $c_2 > 0$ such that 
\begin{equation} \label{eq:ghost}
c_2 \le M^j(t,k,\eta,l) \le 1.
\end{equation}
From (\ref{eq:ghost}), we see that the multiplier $M$ is essentially a Fourier side analogue of Alinhac's ghost energy method for quasilinear wave equations \cite{Alinhac01}, which is the origin of the terminology ``ghost multiplier.'' The multipliers $M_1$ and $M_2$ are used to quantify the inviscid damping with time integrated estimates that do not lose regularity; see for example the first term in (\ref{eq:thrminviscid}) and compare with the pointwise estimate (\ref{eq:lineardamping}). Moreover, they are useful to control terms arising from the linear pressure. The multiplier $M_3$ is designed to balance the transient slow down of the enhanced dissipation that occurs near the critical times. This is quantified by the following lemma.
\begin{lemma}\label{ghostenhanced}
There exists a universal constant $c_3 > 0$ such that for $k \neq 0$ there holds
$$ c_3 \nu^{1/6} \le \nu^{1/2}|k,\eta - kt,l| + \sqrt{-\dot{M_3}M_3}.$$
\end{lemma}
\noindent Using Lemma~\ref{ghostenhanced} we can obtain both pointwise and $L^2$ in time enhanced dissipation estimates that agree with the scaling suggested by the linear theory. See for example the proof of Lemma~\ref{lemma:boot}.

\subsection{Bootstrap argument} \label{outline:bootstrap}

To prove Theorem~\ref{thrm:main} we will use a bootstrap argument. We begin with a statement on the local well-posedness of (\ref{eq:origcoords}).

\begin{lemma} \label{lemma:lwp}
Let $\mu = \nu$, $s_0 > 5/2$, and suppose that $u_{in}$,$b_{in} \in H^s$ for some $s > 7/2$ are divergence free. Then, there exists $T_0(\|(u_{in},b_{in})\|_{H^{s_0}}) > 0$ (in particular, independent of $\nu$) with $\lim_{x\to 0}T_0(x) = \infty$ and a unique classical solution $(u,b) \in C([0,T_0];H^s)$ to (\ref{eq:origcoords}). Moreover, for all $0 < t < T_0$ the solution satisfies
\begin{subequations}
\begin{align}
& \|u(t)\|_{H^{s'}} + \|b(t)\|_{H^{s'}} < \infty \text{ for all }s'\ge 0, \label{eq:lwplemma1} \\ 
& \|\del u\|_{L^2(0,t;H^s)} + \|\del b\|_{L^2(0,t;H^s)} < \infty \label{eq:lwplemma2}. 
\end{align}
\end{subequations}
If $(u,b) \in C([0,T^*);H^s)$ is the maximally extended solution then 
$$\limsup_{t \to T^*}\|(u(t),b(t))\|_{H^{s_0}} = \infty.$$
\end{lemma}
\noindent \textit{Sketch of proof}. The $(X,Y,Z)$ coordinates defined in Sec.~\ref{intro:linearized} are equivalent to the $(x,y,z)$ coordinates for short times in the sense that for any function $g(t,x,y,z) = G(t,X,Y,Z)$ there holds
\begin{equation} \label{eq:coordequiv} \frac{1}{(1+t+t^2)^{s'}}\|G(t)\|^2_{H^{s'}_{X,Y,Z}}\le \|g(t)\|^2_{H^{s'}_{x,y,z}} \le (1+t+t^2)^{s'} \|G(t)\|^2_{H^{s'}_{X,Y,Z}}
\end{equation}
for all ${s'} \ge 0$. Hence, it suffices to prove Lemma~\ref{lemma:lwp} in the new variables. Switching to the new coordinate system and using the unknowns defined in (\ref{eq:elsass1}) and~(\ref{eq:profiles1}), the system (\ref{eq:origcoords}) can be written as
\begin{equation} \label{eq:lwp1}
\begin{cases}
    \dt Z^\pm + \mathbb{P}_t\left((T_{\pm 2\alpha}^t Z^\mp) \cdot \del_L Z^{\pm}\right) + \begin{pmatrix} T_{\pm 2\alpha}^t Z^{\mp,2} \\ 0 \\ 0 \end{pmatrix} = \nu \lap_L Z^{\pm}, \\ 
    \del_L \cdot Z^\pm  = 0,
\end{cases}
\end{equation}
where $\mathbb{P}_t$ denotes the projection onto $\del_L$ divergence free vector fields. Since $\mathbb{P}_t$ satisfies the same properties as the standard Leray projector and $T_{\pm 2\alpha}^t$ is bounded on any $H^{s'}$ space and commutes with $\del_L$, we see that (\ref{eq:lwp1}) has the same energy structure in the nonlinear term as the Navier-Stokes equations. A calculation involving a commutator estimate then yields the a priori bound
\begin{equation} \label{eq:lwp2}
    \frac{d}{dt} \|Z(t)\|_{H^{s_0}}^2 + \nu\|\del_L Z(t)\|_{H^{s_0}}^2  \les (1+t)\|Z(t)\|_{H^{s_0}}^3 + \|Z(t)\|_{H^{s_0}}^2,
\end{equation}
where we have defined the $\R^6$ valued function $Z = (Z^+,Z^-)$. Without loss of generality we can suppose that $t \les 1$, and so estimate (\ref{eq:lwp2}) implies that for some $C > 0$ there holds
\begin{equation} \label{eq:lwp3}
   \|Z(t)\|_{H^{s_0}} \le \frac{\|Z(0)\|_{H^{s_0}}e^{Ct}}{1-\|Z(0)\|_{H^{s_0}}(e^{Ct}-1)}.
\end{equation}
The existence of a unique classical solution $Z \in C([0,T_0];H^s) \cap C^1([0,T_0];H^{s-2})$ to (\ref{eq:lwp1}) for $T_0 \ges \log(1+\|Z(0)\|^{-1}_{H^{s_0}})$ then follows by the classical energy methods used in \cite{MajdaBert} to prove local existence for the 3D Navier-Stokes equations in subcritical Sobolev spaces.
\qed
\par 
\vspace{0.175cm}
\noindent 
A consequence of Lemma~\ref{lemma:lwp}, and in particular (\ref{eq:lwp2}), is that under the assumptions of Theorem~\ref{thrm:main} there exists $0 < t_1 \ll 1$ independent of $\nu$ such that for $c_0$ sufficiently small there holds $\|Z^\pm(t_1)\|_{H^{N+2}} \le \sqrt{2}\epsilon$ and
\begin{equation} \label{eq:boot1}
    \|Z^\pm\|_{L^\infty(0,2t_1;H^{N+2})} + \nu^{1/2}\|\del_L Z^\pm\|_{L^2(0,2t_1;H^{N+2})} \le 4\epsilon.
\end{equation}

\par Recall the definitions of $N$, $N'$, and $N''$ from Theorem~\ref{thrm:main}, and let $\til{N} = N' + 2 + n$. In what follows we use the shorthand notations
\begin{align*}
\lambda(t) & = e^{\delta \nu^{1/3} t}, \\ 
A(t,k,\eta,l) & = m M \lambda, \\ 
\til{A}(t,k,\eta,l) & = \til{m} M \lambda, \\
J(t,k,\eta,l) & = m^{1/2} M \lambda, \\ 
\til{J}(t,k,\eta,l) & = \jap{t}^{-1/2}J. \\
\end{align*}Let $t_2 \ge t_1$ be the maximal time such that the following estimates hold on $[t_1,t_2]$:
\begin{itemize}
\item the high norm bounds:
\end{itemize}
\begin{subequations}\label{eq:highh1}
\begin{align}
\|\til{A}\n{F}^{\pm,1}\|_{L^\infty H^N} + \nu^{1/2}\|\til{A}\del_L\n{F}^{\pm,1}\|_{L^2 H^N} + \|\til{m}\lambda \sqrt{-\dot{M}M}\n{F}^{\pm,1}\|_{L^2 H^N} &\le 8C_0\epsilon \nu^{-1/3} \label{eq:high1},\\ 
\|A\n{F}^{\pm,3}\|_{L^\infty H^N} + \nu^{1/2}\|A\del_L \n{F}^{\pm,3}\|_{L^2 H^N} + \|m\lambda \sqrt{-\dot{M}M}\n{F}^{\pm,3}\|_{L^2 H^N} &\le 8C_0\epsilon \nu^{-1/3} \label{eq:high3},\\ 
\|A \n{H}^2\|_{L^\infty H^N} + \nu^{1/2}\|A \del_L \n{H}^2\|_{L^2 H^N} + \|m \lambda \sqrt{-\dot{M}M}\n{H}^2\|_{L^2 H^N} &\le 8 \epsilon \label{eq:high2nonzero1},\\
\|A \n{Q}^2\|_{L^\infty H^N} + \nu^{1/2}\|A \del_L \n{Q}^2\|_{L^2 H^N} + \|m \lambda \sqrt{-\dot{M}M}\n{Q}^2\|_{L^2 H^N} &\le 8 \epsilon \label{eq:high2nonzero2},\\
\|(H_0,Q_0)\|_{L^\infty H^N} + \nu^{1/2}\|\del (H_0,Q_0)\|_{L^2 H^N} &\le 8 \epsilon \nu^{-1/3}; \label{eq:high2zero}
\end{align}
\end{subequations}
\begin{itemize}
\item the intermediate norm bounds: 
\end{itemize}
\begin{subequations} \label{eq:intboot}
\begin{align}
\|\til{J}\n{F}^{\pm,2}\|_{L^\infty H^{\til{N}}} + \nu^{1/2}\|\til{J}\del_L\n{F}^{\pm,2}\|_{L^2 H^{\til{N}}} + \|\jap{t}^{-1/2} m^{1/2}\lam\sqrt{-\dot{M}M}\n{F}^{\pm,2}\|_{L^2 H^{\til{N}}} &\le 8 \epsilon \label{eq:int1},\\ 
\|J\n{F}^{\pm,2}\|_{L^\infty H^{N'}} + \nu^{1/2}\|J\del_L\n{F}^{\pm,2}\|_{L^2 H^{N'}} + \|m^{1/2}\lam\sqrt{-\dot{M}M}\n{F}^{\pm,2}\|_{L^2 H^{N'}} &\le 8 \epsilon\label{eq:int2} ;
\end{align}
\end{subequations}
\begin{itemize}
\item the low norm bounds:
\end{itemize}
\begin{subequations}
\begin{align}
\|\til{A}\n{F}^{\pm,1}\|_{L^\infty H^{N''}} + \nu^{1/2}\|\til{A}\del_L\n{F}^{\pm,1}\|_{L^2 H^{N''}} + \|\til{m}\lambda \sqrt{-\dot{M}M}\n{F}^{\pm,1}\|_{L^2 H^{N''}} &\le 8C_0\epsilon \label{eq:low1},\\
\|A\n{F}^{\pm,3}\|_{L^\infty H^{N''}} + \nu^{1/2}\|A\del_L \n{F}^{\pm,3}\|_{L^2 H^{N''}} + \|m\lambda \sqrt{-\dot{M}M}\n{F}^{\pm,3}\|_{L^2 H^{N''}} &\le 8C_0\epsilon; \label{eq:low3}
\end{align}
\end{subequations}
\begin{itemize}
\item the zero mode bounds on the velocity and magnetic field:
\end{itemize}
\begin{equation}
\|(u_0,b_0)\|_{L^\infty H^N} + \nu^{1/2}\|\del (u_0,b_0)\|_{L^2 H^N} \le 8 \epsilon\label{eq:zerovelocity}.
\end{equation}
Here, $C_0 \ge 1$ is a constant to be fixed by the proof. We refer to the list of inequalities above as the bootstrap hypotheses. Henceforth, all norms will be taken on $[t_1,t_2]$.
\par 
We claim that $t_2 \ge 2 t_1$ for $t_1$ sufficiently small (still uniformly in $\nu$). Indeed, this follows from (\ref{eq:boot1}), $|\lap_L| \le (1 + t + t^2)|\lap|$, and the fact that all of our Fourier multipliers are continuous and equal to unity at $t = 0$. The plan is then to prove that $t_2 = \infty$. Since all of the norms in (\ref{eq:highh1}) -- (\ref{eq:zerovelocity}) take values continuously in time, it suffices to prove the following proposition.
\begin{proposition} \label{prop:boot}
Under the assumptions of Theorem~\ref{thrm:main}, estimates (\ref{eq:highh1}) -- (\ref{eq:zerovelocity}) all hold with the ``8'' replaced by a ``4'' provided that $t_1 < t_0$ for a sufficiently small universal constant $t_0$.
\end{proposition}
\noindent The proof of Proposition~\ref{prop:boot} is carried out in Secs.~\ref{sec:highmain} -- \ref{sec:zero} and the fact that Proposition~\ref{prop:boot} implies Theorem~\ref{thrm:main} is proven below in Lemma~\ref{lemma:boot}.

\begin{remark} \label{rem:smooth}
The purpose of defining the bootstrap hypotheses on $[t_1,t_2]$ instead of $[0,t_2]$ is to ensure that the classical solution we perform our calculations with satisfies 
\begin{equation} \label{eq:boot2}
\sup_{t\in[t_1,t_2]} \left(\|Z^\pm\|_{H^{s'}} + \|\dt Z^\pm\|_{H^{s'}}\right) < \infty \quad \forall\text{  }s'\ge 0,
\end{equation}
which follows from (\ref{eq:lwplemma1}) and applying Lemma~\ref{lemma:lwp} starting at $t = t_1$.
\end{remark}

\begin{remark}
In light of the discussion at the end of Sec.~\ref{sec:ibpt}, the general structure of the bootstrap hypotheses should be expected. Perhaps the most subtle aspect is the inclusion of $\til{m}$ in the norm for $\n{F}^1$. Physically, this represents allowing the frequencies of $\n{F}^1$ to grow indefinitely after the critical time, which enables us to use integration by parts in time to control the lift-up effect in the low norm with no losses. The key inequality here is (\ref{mtildecay}). On the other hand, the use of a second intermediate norm is more of a technical detail than something deep, and arises essentially from the same scaling that forces one to take $\theta > 0$ in (\ref{eq:model5}).
\end{remark}

\subsubsection{Choice of constants} \label{sec:constants}
Recall the definitions of $c$ and $c_0$ from the statement of Theorem~\ref{thrm:main}. In the proof the various constants will be fixed as follows. We first fix $C_0$ to be a sufficiently large universal constant and $\delta > 0$ to be sufficiently small. Then, $\alpha$ and $c_0$ are chosen to satisfy $\alpha \gg C_0/c$ and $c_0 \ll (\delta/C_0)^p$ for $p$ sufficiently large. We pick $t_0$ in Proposition~\ref{prop:boot} such that $e^{2\delta \nu^{1/3}t_0}(1 + t_0 + t_0^2)^2 \le 2.$ \par

\subsection{Estimates following from the bootstrap hypotheses}

Now we prove a lemma that details the enhanced dissipation and inviscid damping estimates that follow immediately from the bootstrap hypotheses. 

\begin{lemma}\label{lemma:boot}
Let $G$ denote either $Q$ or $H$, and $V$ denote either $U$ or $B$. Under the bootstrap hypotheses the following estimates hold on $[t_1,t_2]$:
\begin{itemize}
\item the enhanced dissipation of $\n{Q}$ and $\n{H}$:
\end{itemize}
\begin{subequations} \label{lem:enh}
\begin{align}
\nu^{1/3}\|\til{A} \n{G}^1\|_{L^2 H^{N}} + \|\til{A}\n{G}^1\|_{L^2 H^{N''}} & \les \epsilon \nu^{-1/6} \\ 
\|A \n{G}^2\|_{L^2 H^{N}} + \|J\n{G}^2\|_{L^2 H^{N'}} & \les \epsilon \nu^{-1/6} \\ 
\nu^{1/3}\|A \n{G}^3\|_{L^2 H^{N}} + \|A\n{G}^3\|_{L^2 H^{N''}} & \les \epsilon \nu^{-1/6}
\end{align}
\end{subequations}
\begin{itemize}
\item the bounds on $\n{U}$ and $\n{B}$, denoting $j\in \{1,3\}$:
\end{itemize}
\begin{subequations} \label{lem:13}
\begin{align}
\nu^{1/3}\|e^{\delta \nu^{1/3}t}\lap_{X,Z}\n{V}^j\|_{L^\infty H^{N}} + \nu^{5/6}\|e^{\delta \nu^{1/3}t}\del_L \lap_{X,Z} \n{V}^j\|_{L^2 H^N} + \nu^{1/2}\|e^{\delta \nu^{1/3}t}\lap_{X,Z}\n{V}^j\|_{L^2 H^N} & \les \epsilon \\ 
\|e^{\delta \nu^{1/3}t}\lap_{X,Z}\n{V}^j\|_{L^\infty H^{N''}} + \nu^{1/2}\|e^{\delta \nu^{1/3}t}\del_L \lap_{X,Z}\n{V}^j\|_{L^2 H^{N''}} + \nu^{1/6}\|e^{\delta \nu^{1/3}t}\lap_{X,Z}\n{V}^j\|_{L^2 H^{N''}} & \les \epsilon \\ 
\|e^{\delta \nu^{1/3}t}\lap_{X,Z}\n{V}^2\|_{L^\infty H^N} + \nu^{1/2}\|e^{\delta \nu^{1/3}t} \del_L \lap_{X,Z}\n{V}^2\|_{L^2 H^N} + \nu^{1/6}\|e^{\delta \nu^{1/3}t}\lap_{X,Z}\n{V}^2\|_{L^2 H^N} &\les \epsilon
\end{align}
\end{subequations}
\begin{itemize}
\item the inviscid damping of $\n{U}^2$ and $\n{B}^2$:
\end{itemize}
\begin{equation}
\|e^{\delta \nu^{1/3}t}\del_{X,Z}\n{V}^2\|_{L^2 H^{N'}} + \|e^{\delta \nu^{1/3}t} \jap{t} \del_{X,Z}\n{V}^2\|_{L^\infty H^{N' - 1}} \les \epsilon. \label{eq:invdamp}
\end{equation}
\end{lemma}
\begin{proof}
First consider the estimates in (\ref{lem:enh}). Observe that for any $s \ge 0$ and $G \in H^s$ we have, by Lemma~\ref{ghostenhanced},
$$\nu^{1/6}\|\n{G}\|_{H^s} \les \nu^{1/2}\|\del_{L} \n{G}\|_{H^s} + \|\sqrt{-\dot{M}M} \n{G}\|_{H^s}.$$
The inequalities in (\ref{lem:enh}) then follow immediately from the bootstrap hypotheses. The estimates in (\ref{lem:13}) follow similarly after employing also (\ref{mlapl}). Now we turn to the inviscid damping estimates. For the first term in (\ref{eq:invdamp}) we use that $|\del_{X,Z}|\les m^{1/2} |\del_L| \les m^{1/2}\sqrt{-\dot{M}M}|\lap_L|$ to obtain $$\|e^{\delta \nu^{1/3}t} \del_{X,Z} \n{V}^2\|_{L^2 H^{N'}} \les \|m^{1/2}\lam \sqrt{-\dot{M}M} \n{G}^2\|_{L^2 H^{N'}},$$
and hence the desired inequality follows from the bootstrap hypothesis (\ref{eq:int2}). For the other term in (\ref{eq:invdamp}) we use $|\del_{X,Z}| \les m^{1/2}|\del_L|^{-1}|\lap_L|$ along with $|\del_L|^{-1} \les \jap{t}^{-1}\jap{\del}$ to derive
$$ \|e^{\delta \nu^{1/3}t} \del_{X,Z} \n{V}^2\|_{H^{N'-1}} \les \|\lam m^{1/2} |\del_L|^{-1} \n{G}^2\|_{H^{N' -1}} \les \jap{t}^{-1} \|J \n{G}^{2}\|_{H^{N'}}, $$
and so the result follows again from (\ref{eq:int2}).
\end{proof}
\noindent We will use the enhanced dissipation estimates in Lemma~\ref{lemma:boot} so frequently throughout the proof that we will typically do so without any remark. \par 

\section{High norm energy estimates} \label{sec:highmain}
Before proceeding to the estimates we establish some simplifying notation to keep the formulas looking as concise as possible. As noted in Remark~\ref{remark:spacetimeres}, our proof does not rely on the non-resonance structure of the nonlinearity.  We will thus systematically drop the transport operator $T_{\pm 2\alpha}^t$ in the nonlinear terms, beyond writing out the initial energy estimate. This is inconsequential because $T_a^t$ commutes with derivatives and preserves norms on $H^s$ spaces. Similarly, it is irrelevant in the nonlinearity which variables are ``$+$'' type and which are ``$-$'' type, and so in the nonlinear terms we will simply drop this superscript. Lastly, by the symmetry of (\ref{eq:reform1}) -- (\ref{eq:reform3}) and (\ref{eq:reform4}) it clearly suffices to estimate only the ``$+$'' variables. \par 

\begin{remark} The weighted energy estimates in the following sections are best understood as being performed on the Fourier side. Note however that the multiplier $m$ is not $C^1$ in time and the a priori bounds on the solution are not enough to ensure that its Fourier transform is continuous. To make the estimates rigorous we mollify $m$ in time, approximate the solution by using a smooth cutoff in the $Y$ variable, and then pass to the limit. This procedure yields the same estimates as one would obtain from a formal calculation because the weak derivative of $m$ is uniformly bounded in time and frequency. For simplicity we omit these steps in the computations.
\label{sec:high}
\end{remark}
\subsection{Estimate of $\n{F}^{\pm,1}$} \label{sec:high1}
In this section we improve (\ref{eq:high1}). Recall the shorthands defined in Sec.~\ref{outline:reformulation}. An energy estimate gives
\begin{align*}
& \frac{1}{2}\|\til{A} \n{F}^{+,1}(t_2)\|^2_{H^N} 
+ \nu\|\til{A}\del_L \n{F}^{+,1}\|_{L^2 H^N}^2 
+ \|\til{m} \lam \sqrt{-\dot{M}M}\n{F}^{+,1}\|_{L^2 H^N}^2 
+ \|M\lam \sqrt{-\dot{\til{m}}\til{m}}\n{F}^{+,1}\|_{L^2 H^N}^2\\ 
& = \frac{1}{2}\|\til{A}\n{F}^{+,1}(t_1)\|_{H^N}^2 -\int_{t_1}^{t_2} \inp{\til{A}\n{F}^{+,1}}{T_{2\alpha}^t\til{A} \n{F}^{-,2}}_{H^N} dt 
- 2\int_{t_1}^{t_2} \inp{\til{A}\n{F}^{+,1}}{\til{A} \dXY^L\lap_L^{-1}\n{F}^{+,1}}_{H^N} dt \\ 
& + \int_{t_1}^{t_2} \inp{\til{A}\n{F}^{+,1}}{\til{A}\partial_{XX} \lap_L^{-1} \n{F}^{+,2}}_{H^N} dt
+ \int_{t_1}^{t_2} \inp{\til{A}\n{F}^{+,1}}{\til{A}\partial_{XX} T_{2\alpha}^t\lap_L^{-1}\n{F}^{-,2}}_{H^N} dt \\
& + \delta \nu^{1/3} \int_{t_1}^{t_2} \|\til{A} \n{F}^{+,1}\|_{H^N}^2 dt
+ \int_{t_1}^{t_2} \inp{\til{A} \n{F}^{+,1}}{\dX \til{A}(\djj^L T_{2\alpha}^tZ^{-,i} \di^L Z^{+,j})}_{H^N}dt \\
& - \int_{t_1}^{t_2} \inp{\til{A} \n{F}^{+,1}}{\til{A}(T_{2\alpha}^t Z^{-}\cdot \del_L F^{+,1})}_{H^N}dt
- \int_{t_1}^{t_2} \inp{\til{A} \n{F}^{+,1}}{\til{A}(T_{2\alpha}^t F^{-}\cdot \del_L Z^{+,1})}_{H^N}dt \\
& - 2\int_{t_1}^{t_2} \inp{\til{A} \n{F}^{+,1}}{\til{A}(T_{2\alpha}^t\di^L Z^{-,j} \dij^L Z^{+,1})}_{H^N}dt \\  
& = \frac{1}{2}\|\til{A}\n{F}^{+,1}(t_1)\|_{H^N}^2 + \te{LU} + \te{LS} + \te{LP1} + \te{LP2} + \te{L}_{\lam} + \te{NLP} + \te{NLT} + \te{NLS1} + \te{NLS2},
\end{align*}
where we have introduced the additional shorthand 
$$\te{L}_{\lambda} = \delta \nu^{1/3} \int_{t_1}^{t_2} \|\til{A} \n{F}^{+,1}\|_{H^N}^2 dt $$
for the term where the time derivative lands on the exponentially growing multiplier $\lambda$. We will continue to use this shorthand throughout for the analogous term in future estimates. The choice of $t_0$ in Sec.~\ref{sec:constants}, along with $\|Z^\pm(t_1)\|^2_{H^{N+2}} \le 2\epsilon^2$, guarantees that $\|\til{A}\n{F}^{+,1}(t_1)\|_{H^N}^2 \le 4 \epsilon^2,$ which is consistent with Proposition~\ref{prop:boot}.

\subsubsection{Lift-up term}
By Cauchy-Schwarz and $|\til{A}| \le |A|$ we have
\begin{align*}
|\te{LU}| \le \|\til{A} \n{F}^{+,1}\|_{L^2 H^N} \|A\n{F}^{-,2}\|_{L^2 H^N} \les \epsilon^2 C_0 \nu^{-1/2} \nu^{-1/6} = C_0^{-1} (\epsilon C_0 \nu^{-1/3})^2,
\end{align*}
which suffices for $C_0$ chosen sufficiently large. 

\subsubsection{Linear stretching term}
It follows by the definition of $\til{m}$ that 
\begin{align*}
\te{LS} \le \|\lam M \sqrt{-\dot{\til{m}}\til{m}} \n{F}^{+,1}\|_{L^2 H^N}^2,
\end{align*}
and so this term is absorbed into the left-hand side of the energy estimate. 

\subsubsection{Linear pressure terms}
Both linear pressure terms are treated similarly, and so we only consider LP1. By (\ref{eq:ghost1}) and $\nu \in (0,1]$ we have 
$$|\te{LP1}| \les \|\til{m} \lam \sqrt{-\dot{M}M} \n{F}^{+,1}\|_{L^2 H^N} \|m \lam \sqrt{-\dot{M}M}\n{F}^{+,2}\|_{L^2 H^N} \les \epsilon^2 C_0 \nu^{-1/3} \le C_0^{-1} (\epsilon C_0 \nu^{-1/3})^2,$$
which is consistent for $C_0$ sufficiently large.

\subsubsection{The term $\te{L}_{\lam}$} \label{sec:lamderiv}
By Lemma~\ref{ghostenhanced} it follows that for $0 < \delta < 1$ sufficiently small there holds 
$$\delta \nu^{1/3} \le \frac{\nu}{2}(k^2 + (\eta - kt)^2 + l^2) - \frac{1}{2} \frac{\dot{M}}{M},$$
from which we obtain
$$L_\lam \le \frac{\nu}{2}\|\til{A}\del_L \n{F}^{+,1}\|_{L^2 H^N}^2 + \frac{1}{2}\|\til{m}\lam  \sqrt{-\dot{M}M}\n{F}^{+,1}\|_{L^2 H^N}.$$
Therefore, $L_\lam$ can be absorbed into the left-hand side of the energy estimate.

\subsubsection{Nonlinear terms} \label{sec:nonlineartermshi1} 
Recall the energy estimate shorthands defined at the end of Sec.~\ref{outline:reformulation}. We begin with the transport term 
$$\text{NLT} = -\int_{t_1}^{t_2} \int \til{A}\jap{\del}^N \n{F}^1 \til{A} \jap{\del}^N (Z \cdot \del_L F^1)_{\neq},$$
where as described above we have dropped the $\pm$ superscripts and the relative transport between the interacting profiles, as they will not be relevant in the nonlinear terms. We will first control the interaction between the nonzero modes. Using $N'' > 3/2$, (\ref{mlapl}), and the paraproduct decomposition defined in Sec.~\ref{sec:freqdecomp}, we have, for $j \in \{1,3\}$,
\begin{align*}
& |\text{NLT}(j,\neq,\neq)| \le |\text{NLT}^{\te{LH}}(j,\neq,\neq)| + |\text{NLT}^{\te{HL}}(j,\neq,\neq)| \\ 
& \quad \quad \les \int_{t_1}^{t_2} \|\til{A}\n{F}^1\|_{H^N}(\|\lam\n{Z}^j\|_{H^{N''}}\|\til{m}^{1/2}\del_L \n{F}^{1}\|_{H^N} +  \|\lam\n{Z}^j\|_{H^{N}}\|\til{m}^{1/2}\del_L \n{F}^{1}\|_{H^{N''}})dt \\
& \quad \quad\les \|\til{A} \n{F}^1\|_{L^\infty H^N} \delta^{-1} \left(\nu^{-1/3}\|\til{A} \n{F}^j\|_{L^2H^{N''}}\|\til{A}\del_L \n{F}^1\|_{L^2 H^N} + \nu^{-1/3}\|\til{A} \n{F}^j\|_{L^2H^{N}}\|\til{A}\del_L \n{F}^1\|_{L^2 H^{N''}}\right)\\ 
& \quad \quad\les  \epsilon^3\delta^{-1}C_0^3\nu^{-1/3}(\nu^{-1/3}\nu^{-1/6}\nu^{-5/6} + \nu^{-1/3}\nu^{-1/2}\nu^{-1/2}) \les (\epsilon C_0 \nu^{-1/3})^2 \epsilon \nu^{-1} C_0 \delta^{-1},
\end{align*}
which suffices for $\epsilon \nu^{-1} \le c_0 \ll \delta C_0^{-1}$.
In the third line above we have used (\ref{mtrivcomt}) and the fact that $t^s e^{-at} \les_s a^{-s}$ for $a \ge 0$ to deduce that for any $s \ge 0$ there holds
\begin{equation} \label{eq:expdec}
1 = \til{m}^{-s}\lam^{-1} \til{m}^{s} \lam \les \jap{t}^{2s}\lam^{-1}\til{m}^s \lam \les \delta^{-2s} \nu^{-2s/3} \til{m}^s \lam.
\end{equation}
Using (\ref{eq:expdec}) as we have done above to compensate for the fact that $\til{m}$ does not satisfy (\ref{mtrivcom}) will be done frequently throughout the proof and is always possible when estimating a term where two nonzero modes interact in the nonlinearity. When we appeal to (\ref{eq:expdec}) in what follows we will typically do so without any remark, and moreover we will not indicate that it causes the underlying constant to depend on an inverse power of $\delta$. We will also no longer show the factors of $C_0$ that appear when estimating the nonlinear terms. In the case $j = 2$ we use $N' - 1 > 3/2$ and the proof of (\ref{eq:invdamp}) to obtain 
\begin{align*}
|\te{NLT}^\te{LH}(2,\neq,\neq)| & \les \int_{t_1}^{t_2}\|\til{A}\n{F}^1\|_{H^N} \|\jap{t} \lam \n{Z}^{2}\|_{H^{N'-1}}\|\jap{t}^{-1}\del_L \n{F}^1\|_{H^N} dt \\ 
& \les \int_{t_1}^{t_2} \|\til{A} \n{F}^1\|_{H^N} \|J\n{F}^2\|_{H^{N'}}\nu^{-1/3}\|\til{A}\del_L \n{F}^1\|_{H^N}dt \\ 
& \les \|\til{A} \n{F}^1\|_{L^\infty H^N} \|J\n{F}^2\|_{L^2 H^{N'}}\nu^{-1/3}\|\til{A} \del_L \n{F}^1\|_{L^2 H^N} \\ 
& \les \epsilon^3 \nu^{-1/3} \nu^{-1/6} \nu^{-1/3}\nu^{-5/6} = (\epsilon \nu^{-1/3})^2 \epsilon \nu^{-1}.
\end{align*}
For the other term in the decomposition, we have
\begin{align*}
|\te{NLT}^\te{HL}(2,\neq,\neq)| & \les \|\til{A}\n{F}^1\|_{L^\infty H^N} \|\lam\n{Z}^2\|_{L^2 H^N} \nu^{-2/3}\|\til{A}\del_L\n{F}^1\|_{L^2 H^{N''}} \\ 
& \les \epsilon^3 \nu^{-1/3}\nu^{-1/6}\nu^{-2/3}\nu^{-1/2} = (\epsilon \nu^{-1/3})^2 \epsilon \nu^{-1},
\end{align*}
which suffices. Now we turn to the interaction between the zero and nonzero modes. For $\te{NLT}(\neq,0)$ we have 
\begin{align*}
|\te{NLT}(\neq,0)| & \les \|\til{A}\n{F}^1\|_{L^\infty H^N}\|\til{A}\n{F}\|_{L^2 H^{N}}\|\del F_0^1\|_{L^2 H^N}\\ 
& \les \epsilon^3\nu^{-1/3}\nu^{-1/2}\nu^{-5/6} = (\epsilon \nu^{-1/3})^2 \epsilon \nu^{-1},
\end{align*}
while for $\te{NLT}(0,\neq)$ we apply (\ref{mcom}) to obtain
\begin{align*}
|\te{NLT}(0,\neq)| & \les \|\til{A} \n{F}^1\|_{L^2 H^N} \|Z_0\|_{L^\infty H^{N+2}}\|\til{A} \del_L \n{F}^1\|_{L^2 H^N} \\ 
& \les \epsilon^3 \nu^{-1/2} \nu^{-1/3} \nu^{-5/6} = (\epsilon \nu^{-1/3})^2 \epsilon \nu^{-1},
\end{align*}
where we combined bootstrap hypotheses (\ref{eq:high2zero}) and (\ref{eq:zerovelocity}) to deduce the bound on $\|Z_0\|_{H^{N+2}}$. This completes the nonlinear transport estimate. \par 
Next we consider 
$$\te{NLS1}(j) = -\int_{t_1}^{t_2} \int \til{A} \jap{\del}^N \n{F}^1 \til{A} \jap{\del}^N(F^j \djj^L Z^1)_{\neq}dV dt$$
and
$$\te{NLS2}(i,j) = -2\int_{t_1}^{t_2} \int \til{A} \jap{\del}^N\n{F}^1\jap{\del}^N \til{A}(\di^L Z^j \dij^L Z^1)_{\neq}dV dt.$$
We start with NLS1, and as before we begin with the interaction between the nonzero modes. When $j \in \{1,3\}$, we use (\ref{eq:expdec}) and (\ref{mlapl}) to obtain
\begin{align*}
& |\te{NLS1}(j,\neq,\neq)| \le |\te{NLS1}^\te{LH}(j,\neq,\neq)| + |\te{NLS1}^\te{HL}(j,\neq,\neq)|\\
& \quad \quad  \les \|\til{A} \n{F}^1\|_{L^\infty H^N}(\nu^{-2/3}\|\til{A}\n{F}^j\|_{L^2 H^{N''}}\|\til{A}\n{F}^1\|_{L^2 H^N} + \nu^{-2/3}\|\til{A}\n{F}^j\|_{L^2 H^{N}}\|\til{A} \n{F}^1\|_{L^2 H^{N''}}) \\ 
& \quad \quad \les \epsilon^3 \nu^{-1/3} \nu^{-2/3} \nu^{-1/6} \nu^{-1/2} = (\epsilon \nu^{-1/3})^2 \epsilon \nu^{-1},
\end{align*}
which suffices. When $j = 2$ we have 
\begin{align*}
& |\te{NLS1}(2,\neq,\neq)| \le |\te{NLS1}^\te{LH}(2,\neq,\neq)| + |\te{NLS1}^\te{HL}(2,\neq,\neq)| \\ 
&\quad \quad \les \|\til{A}\n{F}^1\|_{L^\infty H^N}(\|\n{F}^2\|_{L^2 H^{N'}}\|\lam \dY^L \n{Z}^1\|_{L^2 H^N} + \|\n{F}^2\|_{L^2 H^{N}}\|\lam \dY^L \n{Z}^1\|_{L^2 H^{N''}}) \\ 
& \quad \quad \les \|\til{A}\n{F}^1\|_{L^\infty H^N}(\nu^{-1/3}\|J\n{F}^2\|_{L^2 H^{N'}}\nu^{-1/3}\|\til{A} \n{F}^1\|_{L^2 H^N} + \nu^{-2/3}\|A\n{F}^2\|_{L^2 H^{N}}\nu^{-1/3}\|\til{A} \n{F}^1\|_{L^2 H^{N''}}) \\
& \quad \quad \les \epsilon^3 \nu^{-1/3}(\nu^{-1/3}\nu^{-1/6}\nu^{-1/3}\nu^{-1/2} + \nu^{-2/3}\nu^{-1/6}\nu^{-1/3}\nu^{-1/6}) \les (\epsilon \nu^{-1/3})^2 \epsilon \nu^{-1}.
\end{align*}
\noindent Now we turn to $\te{NLS1}(j,\neq,0)$, which is only nonzero for $j \in \{2,3\}$. In both of these cases we can control the term with (\ref{mcom}), $\til{m} \le m$, and (\ref{mtrivcom}). Indeed, we have
\begin{align*}
|\te{NLS1}(j,\neq,0)| & \les \|\til{A}\n{F}^1\|_{L^2 H^N} \|\lam m^{1/2} \n{F}^j\|_{L^2 H^N} \|Z_0^1\|_{L^\infty H^{N+2}}\\ 
& \les \|\til{A}\n{F}^1\|_{L^2 H^N} \nu^{-1/3} \|A\n{F}^j\|_{L^2 H^N} \|Z^1_0\|_{L^\infty H^{N+2}} \\ 
& \les \epsilon^3 \nu^{-1/2}\nu^{-1/3}\nu^{-1/2} \nu^{-1/3} = (\epsilon \nu^{-1/3})^2 \epsilon \nu^{-1}.
\end{align*}
Lastly, we have 
$$ |\te{NLS1}(0,\neq)| \les \|\til{A}\n{F}^1\|_{L^2 H^N} \|F_0\|_{L^\infty H^N}\|\til{A}\del_L\n{F}^1\|_{L^2 H^N} \les \epsilon^3 \nu^{-1/2} \nu^{-1/3}\nu^{-5/6} = (\epsilon \nu^{-1/3})^2 \epsilon \nu^{-1}, $$
which completes the estimate of NLS1. For NLS2, the interactions between the nonzero modes can be treated in the same manner as they were for NLS1 by using (\ref{eq:expdec}), (\ref{mlapl}), and paraproduct decompositions. We thus skip these terms for the sake of brevity. Turning then to the interaction between the zero and the nonzero modes, we have
\begin{align*}
    |\te{NLS}2(i,j,\neq,0)| &\les \|\til{A}\n{F}^1\|_{L^2 H^N} \nu^{-1/3}\|A\n{F}^j\|_{L^2 H^N} \|Z_0^1\|_{L^\infty H^{N+2}} \\ 
    & \les \epsilon^3 \nu^{-1/2}\nu^{-1/3}\nu^{-1/2}\nu^{-1/3} = (\epsilon \nu^{-1/3})^2 \epsilon \nu^{-1}, 
    \end{align*}
where we noted that we can apply (\ref{mtrivcom}) since the term is only nonzero for $j \neq 1$. Lastly, we have, using (\ref{mcom}), 
\begin{align*}
|\te{NLS2}(i,j,0,\neq)| & \les \|\til{A} \n{F}^1\|_{L^2 H^N}\|Z_0\|_{L^\infty H^{N+2}}\|\til{m}^{1/2}\lam \n{F}^1\|_{L^2 H^N} \\ 
& \les \|\til{A} \n{F}^1\|_{L^2 H^N}\|Z_0\|_{L^\infty H^{N+2}}\|\til{A}\del_L\n{F}^1\|_{L^2 H^N} \\ 
& \les \epsilon^3 \nu^{-1/2}\nu^{-1/3}\nu^{-5/6} = (\epsilon \nu^{-1/3})^2 \epsilon \nu^{-1},
\end{align*}
which completes the treatment of the nonlinear stretching terms.\par
Now we consider the nonlinear pressure, which by an integration by parts can be written
\begin{align*}
    \te{NLP}(i,j) &= \int_{t_1}^{t_2} \int \jap{\del}^N\til{A} \n{F}^1 \jap{\del}^N\til{A}\dX(\djj^L Z^i \di^L Z^j)_{\neq} dV dt \\ 
    & = -\int_{t_1}^{t_2} \int \dX\jap{\del}^N \til{A} \n{F}^1 \jap{\del}^N\til{A}(\djj^L Z^i \di^L Z^j)_{\neq} dV dt.
\end{align*}
There are three distinct cases to consider: $i,j \in \{1,3\}$, $i = 2$ and $j \neq 2$, and $i = j = 2$. For the first case, we notice that due to the symmetry in the bootstrap hypotheses for $F^1$ and $F^3$ we can assume without loss of generality that $Z^j$ has a nonzero $X$-frequency. Then, we have the estimate 
\begin{align*}
|\te{NLP}(i \in \{1,3\},j \in \{1,3\})| & \les \|\del_L \til{A} \n{F}^1\|_{L^2 H^N}\|\djj^L Z^i\|_{L^\infty H^N}\|\til{A}\n{F}^j\|_{L^2 H^{N}} \\
& \les \|\del_L \til{A} \n{F}^1\|_{L^2 H^N}\left(\|\til{A}\n{F}^i\|_{L^\infty H^N} + \|Z_0^i\|_{L^\infty H^{N+2}}\right) \|\til{A}\n{F}^j\|_{L^2 H^N}\\
& \les \epsilon^3 \nu^{-5/6}\nu^{-1/3}\nu^{-1/2} = (\epsilon \nu^{-1/3})^2 \epsilon \nu^{-1}.
\end{align*} 
Turning now to the case $i = 2$ and $j \neq 2$, we first consider when $j = 3$. Splitting the term between the different frequency interactions gives
\begin{align*}|\te{NLP}(2,3)| & \le |\te{NLP}(2,3,0,\neq)| + |\te{NLP}(2,3,\neq,0)| + |\te{NLP}(2,3,\neq,\neq)| \\ 
& := |\te{NLP}(2,3,0,\neq)| + \te{NLP}(2,3,\neq,\cdot).
\end{align*}
For $\te{NLP}(2,3,0,\neq)$ we use (\ref{mcom}) and (\ref{mlapl}) to obtain
\begin{align*}
    |\te{NLP}(2,3,0,\neq)| &\les \|\del_L \til{A} \n{F}^1\|_{L^2 H^N}\|\dZ Z_0^2\|_{L^\infty H^{N+1}} \|\lam \til{m}^{1/2}\dY^L\n{Z}^3\|_{L^2 H^N} \\ 
    & \les \|\del_L \til{A} \n{F}^1\|_{L^2 H^N}\|Z_0^2\|_{L^\infty H^{N+2}} \|A\n{F}^3\|_{L^2 H^N} \\
    & \les \epsilon^3 \nu^{-5/6}\nu^{-1/3}\nu^{-1/2} = \epsilon \nu^{-1} (\epsilon \nu^{-1/3})^2.
\end{align*}
We then estimate the other two pieces using (\ref{mtrivcom}):
\begin{align*}
    \te{NLP}(2,3,\neq,\cdot) & \les \|\del_L \til{A} \n{F}^1\|_{L^2 H^N}\|A \n{F}^2\|_{L^2 H^N}\left( \|Z_0^3\|_{L^\infty H^{N+2}} + \nu^{-1/3}\|A \n{F}^{3}\|_{L^\infty H^N}\ \right) \\ 
    & \les \epsilon^3 \nu^{-5/6}\nu^{-1/6}\nu^{-2/3} = (\epsilon \nu^{-1/3})^2 \epsilon \nu^{-1}.
\end{align*}
A similar estimate holds for $\te{NLP}(2,1)$ due to the fact that $\te{NLP}(2,1,0,\neq) = 0$. The only variation is that we must use (\ref{eq:expdec}) for the interaction between the nonzero modes because $\til{m}$ does not satisfy (\ref{mtrivcom}). We omit the details. Lastly, we consider the case $i = j = 2$, for which we have the estimate
\begin{align*}
|\te{NLP}(2,2)| & \les \|\del_L \til{A}\n{F}^1\|_{L^2 H^N}\|\lam \dY^L \n{Z}^2\|_{L^2 H^N} \|\dY^L Z^2\|_{L^\infty H^N} \\ 
& \les \|\del_L\til{A}\n{F}^1\|_{L^2 H^N}\nu^{-1/3}\|A\n{F}^2\|_{L^2 H^N}(\|Z_0^2\|_{L^\infty H^{N+2}} + \nu^{-1/3}\|A \n{F}^2\|_{L^\infty H^N}) \\ 
& \les \epsilon^3 \nu^{-5/6} \nu^{-1/3} \nu^{-1/6}\nu^{-1/3} = (\epsilon \nu^{-1/3})^2 \epsilon \nu^{-1},
\end{align*}
which suffices and completes the estimate of $\n{F}^{\pm,1}$ in the high norm. 

\subsection{Estimate of $\n{Q}^2$ and $\n{H}^2$} \label{sec:Hi2nonzero}
In this section we improve (\ref{eq:high2nonzero1}) and~(\ref{eq:high2nonzero2}). In particular, we need to verify that with just $\epsilon \ll \nu$ the high norm controls on $\n{H}^2$ and $\n{Q}^2$, unlike (\ref{eq:high1}) and (\ref{eq:high3}), do not need to lose a factor $\nu^{-1/3}$. This will be possible because in the low norm $\n{H}^2$ and $\n{Q}^2$ only grow linearly in time, as opposed to quadratically like $\n{F}^1$ and $\n{F}^3$. Exploiting the gain of one power of $\jap{t}$ does not require us to dramatically alter the structure of the estimates carried out in Sec.~\ref{sec:nonlineartermshi1}. For the sake of brevity the only nonlinear terms we will estimate in detail are NLS1 and NLP. \par 
Using the cancellation 
$$\int A \jap{\del}^N \n{Q}^2 \db (A \jap{\del}^N \n{H}^2)dV + \int A \jap{\del}^N \n{H}^2 \db (A \jap{\del}^N \n{Q}^2)dV = 0$$
and absorbing the term arising from the time derivative landing on $\lambda$ as in Sec.~\ref{sec:lamderiv}, we derive the energy estimate
\begin{align*}
&\frac{1}{2}\|A \n{Q}^2(t_2)\|_{H^N}^2  + \frac{1}{2}\|A \n{H}^2(t_2)\|_{H^N}^2 + \frac{\nu}{2}\|A\del_L \n{Q}^2\|^2_{L^2H^N} + \frac{\nu}{2}\|A\del_L \n{H}^2\|^2_{L^2H^N} \\ 
& + \|M \lambda \sqrt{-\dot{m}m} \n{H}^2\|^2_{L^2 H^N} + \|M \lambda \sqrt{-\dot{m}m} \n{Q}^2\|^2_{L^2 H^N} \\
& + \frac{1}{2}\|m\lambda \sqrt{-\dot{M}M}\n{Q}^2\|^2_{L^2H^N} + \frac{1}{2}\|m\lambda \sqrt{-\dot{M}M}\n{H}^2\|^2_{L^2H^N}\\ 
& \le \frac{1}{2}\|A\n{Q}^2(t_1)\|_{H^N}^2 + \frac{1}{2}\|A\n{H}^2(t_1)\|_{H^N}^2 -2\int_{t_1}^{t_2} \inp{A\n{H}^2}{A\dXY^L\lap_L^{-1}\n{H}^2}_{H^N} dt \\ 
& - \int_{t_1}^{t_2}\inp{A\n{H}^2}{A(U\cdot \del_L H^2 - B \cdot \del_L Q^2)}_{H^N}dt -\int_{t_1}^{t_2} \inp{A\n{H}^2}{A(Q\cdot \del_L B^2 - H \cdot \del_L U^2)}_{H^N} dt \\
& -2\int_{t_1}^{t_2} \inp{A\n{H}^2}{A(\di^L U^j \dij^L B^2 - \di^L B^j \dij^L U^2)}_{H^N}dt + \text{similar terms from }\inp{A\n{Q}^2}{A\dt Q^2}_{H^N}\\ 
& + \int_{t_1}^{t_2} \inp{A\n{Q}^2}{A\dY^L (\djj^L U^i \di^L U^j - \djj^L B^i \di^L B^j)}_{H^N}dt \\ 
& = \frac{1}{2}\|A\n{Q}^2(t_1)\|_{H^N}^2 + \frac{1}{2}\|A\n{H}^2(t_1)\|_{H^N}^2 + \te{LS} + \te{NLT} + \te{NLS1} + \te{NLS2} + \te{``similar terms''} + \te{NLP}, 
\end{align*}
where for the nonlinear stretching and transport terms we have written, for example, $\te{NLT}$ to refer to both $-U \cdot \del_L H^2$ and $B \cdot \del_L  Q^2$ (each abbreviation above is given its own integral sign). Moreover, ``similar terms from $\inp{A \n{Q}^2}{A\dt Q^2}_{H^s}$'' corresponds to, excluding the nonlinear pressure term that is written in the second to last line above, the nonlinear terms arising in an $H^N$ energy estimate of $A \n{Q}^2$. Since $H^2$ and $Q^2$ satisfy the same estimates, it follows that these similar terms can all be controlled following the same methods that we employ below in Sec.~\ref{sec:Hi2nl}, and so we will omit them.  \par 

\subsubsection{Linear stretching term} \label{sec:LSH2Hi}
From the definition of $m$ we have $$\te{LS} \le \|\sqrt{-\dot{m}m}M\lam \n{H}^2\|_{L^2 H^N}^2 + \frac{\nu}{32}\|\del_L A \n{H}^2\|_{L^2 H^N},$$
and hence the linear stretching term can be absorbed into the left-hand side of the energy estimate at the cost of changing the factor of $1/2$ multiplying $\nu\|A \del_L \n{H}^2\|_{L^2 H^N}^2$ into a $15/32$. This is consistent with improving (\ref{eq:high2nonzero1}).

\subsubsection{Nonlinear terms} \label{sec:Hi2nl}
As described above, the only nonlinear terms that we estimate in detail are NLS1 and NLP. We begin with the stretching term and observe that since $Q$ and $H$ satisfy the same estimates it suffices to simply write
$$\te{NLS1} = -\int_{t_1}^{t_2} \int \DN A\n{H}^2 \DN A(Q \cdot \del_L B^2)dV dt.$$
We first consider the interaction between the nonzero modes. When $j \in \{1,3\}$ we have
\begin{align*}
|\te{NLS1}^\te{LH}(j,\neq,\neq)| & \les \|A \n{H}^2\|_{L^\infty H^N}\nu^{-2/3}\|\til{A}\n{Q}^j\|_{L^2 H^{N''}}\|A \n{H}^2\|_{L^2 H^N} \\ 
& \les \epsilon^3 \nu^{-2/3} \nu^{-1/6} \nu^{-1/6} = \epsilon^3 \nu^{-1},
\end{align*}
and, using $N' - 1 > 3/2$, (\ref{mtrivcomt}), and $|\jap{t}\del_{X,Z}| \les \jap{\del} m^{1/2} |\lap_L|$ ,
\begin{align*}
|\te{NLS1}^\te{HL}(j,\neq,\neq)| &\les \int_{t_1}^{t_2} \|A \n{H}^2\|_{H^N} \|\jap{t}^{-1}\n{Q}^j\|_{H^{N}}\|\lam \jap{t} \djj^L  \n{B}^2\|_{H^{N' - 1}} dt \\ 
& \les \|A \n{H}^2\|_{L^\infty H^N} \nu^{-1/3}\|\til{A}\n{Q}^j\|_{L^2 H^N} \|J \n{H}^2\|_{L^2 H^{N'}} \\ 
& \les \epsilon^3 \nu^{-1/3}\nu^{-1/2} \nu^{-1/6} = \epsilon^3 \nu^{-1}.
\end{align*}
For $j = 2$ we use to (\ref{mtrivcom}) to obtain 
\begin{align*}
|\te{NLS1}(2,\neq,\neq)| &\le |\te{NLS1}^\te{LH}(2,\neq,\neq)| + |\te{NLS1}^\te{HL}(2,\neq,\neq)| \\ 
& \les \|A \n{H}^2\|_{L^\infty H^N}\left(\|\n{Q}^2\|_{L^2 H^{N'}}\|\lam \dY^L \n{B}^2\|_{L^2 H^N} + \|\n{Q}^2\|_{L^2 H^{N}}\|\lam \dY^L \n{B}^2\|_{L^2 H^{N'}}\right) \\ 
& \les \|A \n{H}^2\|_{L^\infty H^N}\nu^{-2/3}\left(\|J\n{Q}^2\|_{L^2 H^{N'}}\|A \n{H}^2\|_{L^2 H^N} + \|A\n{Q}^2\|_{L^2 H^{N}}\|J \n{H}^2\|_{L^2 H^{N'}}\right) \\
& \les \epsilon^3\nu^{-2/3}\nu^{-1/6}\nu^{-1/6} = \epsilon^3 \nu^{-1}.
\end{align*}
Now we turn to $\te{NLS1}(j,\neq,0)$, which is only nonzero when $j \neq 1$. We have, from (\ref{mtrivcom}) and (\ref{mcom}), 
\begin{align*}
|\te{NLS1}(j,\neq,0)| & \le |\te{NLS1}^\te{HL}(j,\neq,0)|  + |\te{NLS1}^\te{LH}(j,\neq,0)| \\
& \les \|A \n{H}^2\|_{L^2 H^N} \nu^{-1/3}\left(\|A\n{Q}^j\|_{L^2 H^N} \|B^2_0\|_{L^\infty H^{N}} +  \|A\n{Q}^j\|_{L^2 H^{N''}} \|B^2_0\|_{L^\infty H^{N+2}}\right)\\ 
& \les \epsilon^3 \nu^{-1/6}\nu^{-1/3}(\nu^{-1/2} + \nu^{-1/6}\nu^{-1/3}) = \epsilon^3 \nu^{-1}.
\end{align*}
The fact that $\te{NLS}(1,\neq,0) = 0$ is crucial, as the method just used for $j \in \{2,3\}$ would fail since we would be forced to bound $\|m\n{F}^1\|_{H^N}$, which we do not have control on due to $m \ge \til{m}$. Lastly, for $\te{NLS}1(0,\neq)$ there holds
$$ |\te{NLS1}(0,\neq)| \les \|A\n{H}^2\|_{L^2 H^N} \|Q_0\|_{L^\infty H^N} \nu^{-1/3}\|A\n{H}^2\|_{L^2 H^N} \les \epsilon^3 \nu^{-1/6} \nu^{-1/3}\nu^{-1/3}\nu^{-1/6} = \epsilon^3 \nu^{-1}. $$ 
This completes the estimate of NLS1. \par 
\begin{remark}\label{rem:nonlinearstructure}
The structure that caused $\te{NLS}(1,\neq,0)$ to vanish is not unique to the term. Rather, it comes from the general $Z^j \djj^L$ structure of the nonlinearity that causes $Z^1$ to always be paired with an $X$ derivative. This cancellation is used in many of the $(\neq,0)$ estimates that we omit for $\n{F}^2$ and $\n{F}^3$, and is important for the proof to work.
\end{remark}
We turn now to the nonlinear pressure term. It reads
\begin{align*}
    \te{NLP}(i,j) &= \int_{t_1}^{t_2} \int A\jap{\del}^N \n{H}^2 A\jap{\del}^N\dY^L(\djj^L U^i \di^L U^j)_{\neq} dV dt \\
    & = -\int_{t_1}^{t_2} \int \dY^L A\jap{\del}^N \n{H}^2 A\jap{\del}^N(\djj^L U^i \di^L U^j) dV dt.
\end{align*}
We begin with the interaction between the zero and nonzero modes. By symmetry, it suffices to consider $\te{NLP}(i,j,\neq,0)$, which we note is only nonzero for $i \neq 1$. We have 
\begin{align*}
    |\te{NLP}(i,j,\neq,0)| & \le |\te{NLP}^{\te{HL}}(i,j,\neq,0)| + |\te{NLP}^{\te{LH}}(i,j,\neq,0)| \\ 
    & \les \|A \del_L \n{H}^2\|_{L^2 H^N}\left(\|A \n{Q}^i\|_{L^2 H^N}\|U_0^j\|_{L^\infty H^N} + \|A \n{Q}^i\|_{L^2 H^{N''}}\|U_0^j\|_{L^\infty H^{N+2}}\right) \\ 
    & \les \epsilon^3 \nu^{-1/2}(\nu^{-1/2} + \nu^{-1/6} \nu^{-1/3})  \les \epsilon^3 \nu^{-1},
\end{align*}
which is consistent. Note that in the related papers \cite{BGM15I,WZ18} the term $\te{NLP}(3,1,\neq,0)$ is the leading order piece of the nonlinearity, but for us this term is relatively easy due to the suppression of the lift up effect for $Q^1_0$. For the interaction between the nonzero modes there are, as before, three distinct cases to consider: $i,j \in \{1,3\}$, $i = 2$ and $j \neq 2$, and $i = j = 2$. We first consider when $i,j \in \{1,3\}$. By symmetry, we only need to control the HL interaction. We have
\begin{align*}
|\te{NLP}^\te{HL}(i \in\{1,3\},j\in\{1,3\},\neq,\neq)| & \les \|\del_L A \n{H}^2\|_{L^2 H^N} \|\til{A} \n{Q}^i\|_{L^\infty H^{N}}\|\til{A}\n{Q}^j\|_{L^2 H^{N''}} \\ 
& \les \epsilon^3 \nu^{-1/2}\nu^{-1/3}\nu^{-1/6} = \epsilon^3 \nu^{-1}.
\end{align*}
Turning to the second case, we treat the LH interaction using Lemma~\ref{lemma:boot}, (\ref{mtrivcomt}), and (\ref{mlapl}):
\begin{align*}
|\te{NLP}^\te{LH}(i = 2, j \neq 2, \neq, \neq)| & \les \|\del_L A \n{H}^2\|_{L^2H^N} \|\jap{t}\del_{X,Z}\n{U}^2\|_{L^\infty H^{N'-1}} \|\til{A} \n{Q}^j\|_{L^2H^N} \\ 
& \les \epsilon^3 \nu^{-1/2} \nu^{-1/2} = \epsilon^3 \nu^{-1}.
\end{align*}
For the HL term we have
\begin{align*}
|\te{NLP}^\te{HL}(i = 2,j \neq 2,\neq,\neq)| & \les \|\del_L A \n{H}^2\|_{L^2 H^N}\|A \n{Q}^2\|_{L^\infty H^N} \nu^{-1/3}\|\til{A} \n{Q}^j\|_{L^2 H^{N''}} \\ 
& \les \epsilon^3 \nu^{-1/2}\nu^{-1/3}\nu^{-1/6} = \epsilon^3 \nu^{-1}.
\end{align*}
Lastly we consider $i = j = 2$. It suffices to consider only the HL term, for which we have the estimate
$$|\text{NLP}^\te{HL}(2,2,\neq,\neq)| \les \|A \del_L \n{H}^2\|_{L^2 H^N} \nu^{-1/3}\|A\n{Q}^2\|_{L^2 H^N} \|J \n{Q}^2\|_{L^\infty H^{N'}} \les \epsilon^3 \nu^{-1}.$$
This completes the estimate of $\n{H}^2$ and $\n{Q}^2$ in the high norm.

\subsection{Estimate of $\n{F}^3$} \label{sec:high3}
Improving (\ref{eq:high3}) follows from essentially the same methods used in Sec.~\ref{sec:high1}. To see this, first note that, disregarding the lift-up term, $F^3$ satisfies the same equation as $F^1$ except for the presence of $\dZ$ instead of $\dX$ in the pressure terms. This is inconsequential in the estimates. For example, the linear pressure is simply controlled with (\ref{eq:ghost2}) instead of (\ref{eq:ghost1}). The only other variations are due to the use of $m$ instead of $\til{m}$ in the norm and were already encountered in the estimate of $\n{F}^2$. In particular, we treat the linear stretching term as in Sec.~\ref{sec:LSH2Hi}, and we rely on the crucial nonlinear structure just noted above in Remark~\ref{rem:nonlinearstructure} and the estimate of $\te{NLS}1(\neq,0)$.

\subsection{Summary of high norm nonzero mode interactions}

For the sake of clarity in the remainder of the paper it is useful to gather the above calculations into a general lemma. Let $(\partial_t F^j)_{\mathcal{NL}}$ denote the nonlinear terms in $\dt F^j$. We then also define $(\partial_t F^j)_{\mathcal{NL}}^{\neq \neq}$ and $(\partial_t F^j)_{\mathcal{NL}}^{00}$ to denote $(\partial_t F^j)_{\mathcal{NL}}$ restricted to either the interaction between the nonzero modes or the interaction between the zero modes. \par 
Our estimates of the $(\neq,\neq)$ nonlinear interactions in Secs.~\ref{sec:high1} --~\ref{sec:high3} only relied on the enhanced dissipation of the functions in the particular nonlinear term (i.e., we did not use the enhanced dissipation of the function that plays the role of $G$ in Lemma~\ref{lemhighnonzero} below). Moreover, we did not employ any commutator type estimates for $m$ and $\til{m}$ other than (\ref{mlapl}), (\ref{mtrivcomt}), and (\ref{mtrivcom}). In particular, we did not use (\ref{mcom}). Due to these observations, our calculations above yield the following lemma.

\begin{lemma} \label{lemhighnonzero}
Let $\beta \ge 0$ and suppose that $G$ is a smooth function that satisfies 
$$\|G\|_{L^\infty H^s} + \nu^{1/2}\|\del_L G\|_{L^2 H^s} \les \epsilon \nu^{-\beta} $$
for some $s \le N$. Then, for any bounded Fourier multiplier $\mathcal{M}$ there holds 
\begin{align}
    \left| \int_{t_1}^{t_2} \inp{\lam \mathcal{M}(\dt F^j)_{\mathcal{NL}}^{\neq \neq}}{G}_{H^s} dt\right| & \les \epsilon^3 \nu^{-4/3 - \beta} \quad (j \in \{1,3\}),  \\ 
    \left| \int_{t_1}^{t_2} \inp{\lam \mathcal{M}(\dt F^2)_{\mathcal{NL}}^{\neq \neq}}{G}_{H^s} dt\right| & \les \epsilon^3 \nu^{-1 - \beta }.
\end{align}
\end{lemma}

\subsection{Estimate of $Q_0$ and $H_0$}
In this section we improve (\ref{eq:high2zero}). For any $r \in \{1,2,3\}$, an energy estimate gives 
\begin{align*}
& \frac{1}{2}\|F_0^{+,r}(t_2)\|^2_{H^N}  + \nu \|\del F_0^{+,r}\|_{L^2 H^N}^2  = \frac{1}{2}\|F_0^{+,r}(t_1)\|_{H^N}^2 - \ind_{r = 1} \int_{t_1}^{t_2} \inp{F_0^{+,1}}{T_{2\alpha}^t F_0^{-,2}}_{H^N}dt \\ 
& - \int_{t_1}^{t_2} \inp{F^{+,r}_0} {T_{2\alpha}^t Z^{-} \cdot \del_L F^{+,r}}_{H^N}dt
 - \int_{t_1}^{t_2} \inp{F^{+,r}_0}{T_{2\alpha}^t F^{-} \cdot \del_L Z^{+,r}}_{H^N}dt \\
& - \int_{t_1}^{t_2} \inp{F^{+,r}_0}{T_{2\alpha}^t \di^L Z^{-,j} \dij^L Z^{+,r}}_{H^N}dt 
+ \ind_{r \neq 1}\int_{t_1}^{t_2} \inp{F_0^{+,r}} {\partial_r (T_{2\alpha}^t\djj^L Z^{-,i} \di^L Z^{+,j})}_{H^N}dt \\
& \quad = \frac{1}{2}\|F_0^{+,r}(t_1)\|_{H^N}^2 + \te{LU} + \te{NLT} + \te{NLS1} + \te{NLS2} + \te{NLP}.
\end{align*}

\subsubsection{Nonlinear terms} \label{sec:nonlinearhizero}
For the interaction between the nonzero modes we have, by Lemma~\ref{lemhighnonzero} and (\ref{eq:high2zero}), 
$$\left|\int_{t_1}^{t_2} \inp{F_0^r}{(\dt F^r)_{\mathcal{NL}}^{\neq \neq}}_{H^N} dt \right| \les (\epsilon \nu^{-1/3})^2 \epsilon \nu^{-1},$$
which is consistent. It then only remains to consider the interaction between the zero modes. We begin with the transport term
$$\te{NLT}(j,0,0) = -\int_{t_1}^{t_2} \int \DN F_0^{+,r} \DN (Z_0^{-,j}\djj F_0^{+,r})dV dt.$$
When $j = 2$ we use that incompressibility implies that $Z_0^2$ always has a nonzero $Z$-frequency to obtain 
$$|\te{NLT}(2,0,0)| \les \|F_0^r\|_{L^\infty H^N} \|\del Z_0^2\|_{L^2 H^N} \|\del F_0^r\|_{L^2 H^N} \les \epsilon^3 \nu^{-1/3} \nu^{-1/2} \nu^{-5/6} = (\epsilon \nu^{-1/3})^2 \epsilon \nu^{-1}.$$
For $j = 3$ we observe that the term vanishes unless at least one of $Z_0^3$ or $F_0^r$ has a nonzero $Z$-frequency. Hence,
\begin{align*}
|\te{NLT}(3,0,0)| &\les \|F_0^r\|_{L^\infty H^N} \|\del Z_0^3\|_{L^2 H^N} \|\dZ F_0^r\|_{L^2 H^N} + \|\del F_0^r\|_{L^2 H^N} \|Z_0^3\|_{L^\infty H^N} \|\dZ F_0^r\|_{L^2 H^N} \\ 
& \les \epsilon^3(\nu^{-1/3}\nu^{-1/2}\nu^{-5/6} + \nu^{-5/6}\nu^{-5/6}) \les (\epsilon \nu^{-1/3})^2 \epsilon \nu^{-1},
\end{align*}
which suffices and completes the estimate of the transport nonlinearity. Using incompressibility and $\dZ = \dZ P_{l \neq 0}$ in a similar fashion as above, both of the stretching terms are treated in essentially the same way as the transport term.
We thus skip them and turn to the pressure, for which, after an integration by parts, we have the estimate
$$|\te{NLP}(i,j,0,0)| \les \|\del F_0^r\|_{L^2 H^N} \|\del Z^j_0\|_{L^2 H^N} \|Z^i_0\|_{L^\infty H^{N+2}} \les \epsilon^3 \nu^{-5/6}\nu^{-1/2}\nu^{-1/3} = (\epsilon \nu^{-1/3})^2 \epsilon \nu^{-1}.$$
This completes the high norm estimate of the nonlinear terms for $F_0^r$, $r \in \{1,2,3\}$. \par  

The computations in this section were not sensitive to the component of $F_0$ being estimated. In fact, the estimates of NLT, NLS1, and NLS2, which are each quadratic in the $r$ component, do not even rely on any structures that would cause the same methods to fail if the two occurrences of $r$ were replaced by $r$ and some $r'$. This generality gives us the following lemma, which will be useful when considering the nonlinear terms that arise in our treatment of the lift-up effect with integration by parts in time.

\begin{lemma} \label{lemhighzero}
Let $r \in \{1,2,3\}$ and suppose that $G$ is a smooth function that satisfies
$$\|G\|_{L^\infty H^N} + \nu^{1/2}\|\del G\|_{L^2 H^N} \les \epsilon \nu^{-1/3}.$$
Then, there holds
$$\left|\int_{t_1}^{t_2} \inp{G}{(\dt F^{r})_{\mathcal{NL}}^{00}}_{H^N}dt\right| \les \epsilon^3 \nu^{-5/3}.$$
\end{lemma}

\subsubsection{Suppression of the lift-up effect} \label{sec:luhi}
As discussed above, the main stabilizing effect in our work is that the magnetic field induces oscillations that suppress the lift-up effect. In this section, we show how to estimate LU with no losses by exploiting these oscillations using integration by parts in time. Of crucial importance is that incompressibility implies that $F_0^{-,2}$ always has a nonzero $Z$-frequency, which ensures that there is no component of the lift-up term that does not oscillate. Noting that $T_{\alpha}^t g_0 = e^{\alpha t \dZ} g_0$ for any function $g$, we integrate by parts in time to obtain 
\begin{align}
& - \te{LU}  = \frac{1}{2\alpha}\int_{t_1}^{t_2}\int \DN \pz F_0^{+,1} \DN \dZ^{-1} \pz \dt T_{2\alpha}^t F_0^{-,2} dV dt \nonumber \\ 
& = \frac{1}{2\alpha}\inp{\pz F_0^{+,1}(t_2)}{\dZ^{-1}T_{2\alpha}^{t_2}\pz F_0^{-,2}(t_2)}_{H^N} - \frac{1}{2\alpha}\inp{\pz F_0^{+,1}(t_1)}{\dZ^{-1}T_{2\alpha}^{t_1}\pz F_0^{-,2}(t_1)}_{H^N} \label{eq:LUboundary}\\ 
& - \frac{1}{2\alpha}\int_{t_1}^{t_2} \inp{\pz \dt F_0^{+,1}}{\dZ^{-1}T_{2\alpha}^t\pz F_0^{-,2}}_{H^N}dt - \frac{1}{2\alpha}\int_{t_1}^{t_2} \inp{\pz F_0^{+,1}}{\dZ^{-1}T_{2\alpha}^t\pz \dt F_0^{-,2}}_{H^N}dt. \label{eq:LUdt}
\end{align}
The boundary terms in (\ref{eq:LUboundary}) are both treated similarly. For example, by Cauchy-Schwarz and the fact that $\dZ^{-1}\pz$ is bounded on $H^N$ we have
\begin{align*}
\left|\frac{1}{2\alpha}\inp{\pz F_0^{+,1}(t_2)}{\dZ^{-1}T_{2\alpha}^{t_2}\pz F_0^{-,2}(t_2)}_{H^N}\right| & \les \frac{1}{\alpha}\|F_0^{+,1}\|_{L^\infty H^N} \|F_0^{-,2}\|_{L^\infty H^N} \\ 
& \les \frac{1}{\alpha} (\epsilon \nu^{-1/3})^2,
\end{align*}
which is consistent for $\alpha$ sufficiently large. Now we turn to (\ref{eq:LUdt}). Expanding out the first of the two terms gives
\begin{align*}
   & \frac{1}{2\alpha}\int_{t_1}^{t_2} \inp{\pz \dt F_0^{+,1}}{\dZ^{-1}T_{2\alpha}^t\pz F_0^{-,2}}_{H^N}dt
   = \frac{1}{2\alpha}\int_{t_1}^{t_2} \inp{\nu \lap \pz \z{F}^{+,1}}{T_{2\alpha}^t\dZ^{-1}\pz F_0^{-,2}}_{H^N}dt \\ 
    & - \frac{1}{2\alpha}\int_{t_1}^{t_2}\inp{T_{2\alpha}^t F_0^{-,2}}{\dZ^{-1} T_{2\alpha}^t \pz F_0^{-,2}}_{H^N} dt 
     + \frac{1}{2\alpha}\int_{t_1}^{t_2} \inp{\pz (\dt F_0^{+,1})_{\mathcal{NL}}}{\dZ^{-1}T_{2\alpha}^t\pz F_0^{-,2}}_{H^N}dt \\ 
     & = \te{LU1} + \te{LU2} + \mathcal{NL}.
\end{align*}
The linear term LU1 arising from the dissipation is treated naturally with an integration by parts:
\begin{align*}
|\te{LU1}| = \left| \frac{1}{2\alpha}\int_{t_1}^{t_2} \inp{\nu \del \pz \z{F}^{+,1}}{T_{2\alpha}^t\dZ^{-1}\del \pz F_0^{-,2}}_{H^N}dt \right|
& \les \frac{\nu}{\alpha}\|\del \z{F}^{+,1}\|_{L^2 H^{N}}\|\del \z{F}^{-,2}\|_{L^2 H^{N}} \\ 
 & \les \frac{\nu}{\alpha} \epsilon^2 \nu^{-5/6} \nu^{-5/6} = \frac{1}{\alpha} (\epsilon \nu^{-1/3})^2,
\end{align*}
which is consistent for $\alpha$ sufficiently large. Crucially, the other linear term LU2 vanishes. Indeed, the inner product under the time integral can be rewritten as
$$
\int \pz \DN T_{2\alpha}^tF_0^{-,2} \DN \dZ^{-1} T_{2\alpha}^t\pz F_0^{-,2}dV = \frac{1}{2}\int \dZ \left(\dZ^{-1} \pz \DN T_{2\alpha}^t F_0^{-,2}\right)^2 dV = 0.$$ 
For the nonlinear contribution $\mathcal{NL}$ we have, by (\ref{eq:high2zero}) and Lemmas~\ref{lemhighnonzero} and~\ref{lemhighzero},
\begin{align*}
|\mathcal{NL}| & \les \left|\int_{t_1}^{t_2} \inp{\dZ^{-1} T_{2\alpha}^t P_{l\neq 0} F_0^{-,2}}{(\partial_t F^{+,1})_{\mathcal{NL}}^{\neq \neq}}\right| + \left|\int_{t_1}^{t_2} \inp{\dZ^{-1} T_{2\alpha}^t P_{l\neq 0} F_0^{-,2}}{(\partial_t F^{+,1})_{\mathcal{NL}}^{0 0}}\right| \\ 
& \les \epsilon^3 \nu^{-5/3} = (\epsilon \nu^{-1/3})^2 \epsilon \nu^{-1}. 
\end{align*}
The second term in (\ref{eq:LUdt}) is estimated similarly, and so we omit the details. In fact, this term is simpler because the only linear contribution comes from the dissipation.

\section{Intermediate norm energy estimates} \label{sec:int}
Our focus in this section is how to use (\ref{eq:dioph2}) along with the the high norm control on $\n{F}^2$ to improve~(\ref{eq:intboot}). We will provide the details for improving (\ref{eq:int1}) and then briefly discuss how the same techniques carry over to the other estimate.

\subsection{Estimate of $\n{F}^2$ in $H^{N' + 2 + n}$} \label{sec:Ntil}
Recall the notations $\til{N} = N' + 2 + n$ and $\til{J} = \jap{t}^{-1/2} J$. Dropping the negative term from the time derivative landing on the decaying time weight and absorbing $L_\lam$ as before into the left-hand side, we obtain the energy estimate 
\begin{align*}
& \frac{1}{2}\|\til{J} \n{F}^{+,2}(t_2)\|^2_{H^{\til{N}}} + 
\frac{\nu}{2}\|\til{J} \del_L \n{F}^{+,2}\|^2_{L^2 H^{\til{N}}} 
+ \frac{1}{2}\|\jap{t}^{-1/2}m^{1/2}\lam \sqrt{-\dot{M}M} \n{F}^{+,2}\|^2_{L^2 H^{\til{N}}}\\ 
& + \|\jap{t}^{-1/2}M\lam \sqrt{-\dot{m}^{1/2}m^{1/2}}\n{F}^{+,2}\|_{L^2 H^{\til{N}}}^2 \le \frac{1}{2}\|\til{J}\n{F}^{+,2}(t_1)\|_{H^N}^2 -\int_{t_1}^{t_2} \inp{\til{J} \n{F}^{+,2}}{\dXY^L \lap_L^{-1}\til{J} \n{F}^{+,2}}_{H^{\til{N}}} dt \\ 
& + \int_{t_1}^{t_2} \inp{\til{J} \n{F}^{+,2}}{\dXY^L \lap_L^{-1}T_{2\alpha}^t\til{J}\n{F}^{-,2}}_{H^{\til{N}}} dt - \int_{t_1}^{t_2} \inp{\til{J}\n{F}^{+,2}}{\til{J} (T_{2\alpha}^t Z^-\cdot \del_L F^{+,2})}_{H^{\til{N}}} dt \\
& - \int_{t_1}^{t_2} \inp{\til{J}\n{F}^{+,2}}{\til{J} (T_{2\alpha}^t F^{-}\cdot \del_L Z^{+,2})}_{H^{\til{N}}}dt 
- 2\int_{t_1}^{t_2} \inp{\til{J}\n{F}^{+,2}}{\til{J} (T_{2\alpha}^t \di^L Z^{-,j}\dij^L Z^{+,2})}_{H^{\til{N}}}dt \\
& + \int_{t_1}^{t_2} \inp{\til{J}\n{F}^{+,2}}{\dY^L \til{J} (T_{2\alpha}^t \djj^L Z^{-,i}\di^L Z^{+,j})}_{H^{\til{N}}}dt \\
& = \frac{1}{2}\|\til{J}\n{F}^{+,2}(t_1)\|_{H^N}^2 + \text{LS} + \text{OLS} + \text{NLT} + \text{NLS1} + \text{NLS2} + \te{NLP}.
\end{align*}
Below we consider only OLS and the nonlinear terms, since the term LS can absorbed into the left-hand side of the estimate by using (\ref{eq:mhalf}) in the same way that we used (\ref{eq:m}) in Sec.~\ref{sec:LSH2Hi}.

\subsubsection{Oscillating linear stretching term} \label{sec:intlso}
We now use integration by parts in time to control the oscillating linear stretching term with no losses, which is key to the proof and the fundamental difference between the results in Theorem~\ref{thrm:main} and Corollary~\ref{cor:4/3}. We begin by introducing the shorthand notation $S = S(t,\del) = \dXY^L \lap_L^{-1}$. That is, $S$ is the Fourier multiplier with symbol
$$S(t,k,\eta,l) = \frac{k(\eta - kt)}{k^2 + l^2 + (\eta - kt)^2}.$$
Note that we have the inequality 
\begin{equation} \label{eq:Sbound}
|S(t)| \les \frac{1}{\jap{t}} |k,l,\eta|.
\end{equation}
Integrating by parts in time we write the term as
\begin{align}
\te{OLS} & = \frac{1}{2\alpha}\int_{t_1}^{t_2} \inp{\til{J} \n{F}^{+,2}}{\dt T_{2\alpha}^t\db^{-1}S \til{J}\n{F}^{-,2}}_{H^{\til{N}}} dt\nonumber \\ 
& = \frac{1}{2\alpha}\inp{\til{J}\n{F}^{+,2}(t_2)}{T_{2\alpha}^{t_2} \db^{-1} S \til{J} \n{F}^{-,2}(t_2)}_{H^{\til{N}}} -  \frac{1}{2\alpha} \inp{\til{J}\n{F}^{+,2}(t_1)}{T_{2\alpha}^{t_1}\db^{-1} S \til{J} \n{F}^{-,2}(t_1)}_{H^{\til{N}}} \label{int:boundary} \\
& \qquad - \frac{1}{2\alpha} \int_{t_1}^{t_2} \left\langle \til{J}\n{F}^{+,2},\left(2\frac{\dot{\til{J}}}{\til{J}}S + \dot{S}\right)T_{2\alpha}^t\db^{-1}\til{J}\n{F}^{-,2}\right\rangle_{H^{\til{N}}} dt \label{int:dtmult}\\ 
& \qquad- \frac{1}{2\alpha} \int_{t_1}^{t_2} \inp{\til{J}\dt \n{F}^{+,2}}{T_{2\alpha}^t\db^{-1}S\til{J}\n{F}^{-,2}}_{H^{\til{N}}}dt \label{int:dt1}\\
& \qquad- \frac{1}{2\alpha} \int_{t_1}^{t_2} \inp{\til{J} \n{F}^{+,2}}{T_{2\alpha}^t\db^{-1}\til{J}S\dt\n{F}^{-,2}}_{H^{\til{N}}}dt. \label{int:dt2}
\end{align}
The boundary terms in (\ref{int:boundary}) are both treated similarly, and so we will only estimate the one at $t = t_2$. Recalling the definitions of $c$ and $n$ from Theorem~\ref{thrm:main}, we have, by Cauchy-Schwarz, (\ref{eq:dioph2}), (\ref{mtrivcomt}), and  (\ref{eq:Sbound}),
\begin{align*}
\frac{1}{2\alpha}\left|\inp{\til{J}\n{F}^{+,2}(t_2)}{T_{2\alpha}^{t_2} \db^{-1} S \til{J} \n{F}^{-,2}(t_2)}_{H^{\til{N}}}\right| & \les \frac{1}{c \alpha} \|\til{J} \n{F}^{+,2}\|_{L^\infty H^{\til{N}}}\|\jap{t}^{-1/2} m^{-1/2} S A\n{F}^{-,2}\|_{L^\infty H^{\til{N}+n}}
\\
& \les \frac{1}{c \alpha} \|\til{J} \n{F}^{+,2}\|_{L^\infty H^{\til{N}}}\|\jap{t}^{1/2} S A\n{F}^{-,2}\|_{L^\infty H^{\til{N}+n}} \\ 
& \les \frac{1}{c\alpha}\|\til{J} \n{F}^{+,2}\|_{L^\infty H^{\til{N}}}\|A \n{F}^{-,2}\|_{L^\infty H^{\til{N} + 1 + n}} \\ 
& \les \frac{\epsilon^2}{c \alpha},
\end{align*}
which suffices for $c \alpha$ chosen sufficiently large. In the last line above we have used that $\til{N} + 1 + n = N' + 3 + 2n \le N$. Next consider (\ref{int:dtmult}), which splits into five terms since 
\begin{equation} \label{eq:JdotonJ}
\frac{\dot{\til{J}}}{\til{J}}S + \dot{S} = \left(\frac{\dot{m}^{1/2}}{m^{1/2}} + \delta \nu^{1/3} + \frac{\dot{M}}{M} - \frac{1}{2}t\jap{t}^{-2}\right)S + \dot{S}. 
\end{equation}
In the order listed in (\ref{eq:JdotonJ}) we label these five terms as 
$$(\ref{int:dtmult}) = (\ref{int:dtmult}\te{a}) + (\ref{int:dtmult}\te{b}) + (\ref{int:dtmult}\te{c}) + (\ref{int:dtmult}\te{d}) + (\ref{int:dtmult}\te{e}).$$
For (\ref{int:dtmult}\te{a}) we recall (\ref{eq:mhalf}), which states that $|\dot{m}^{1/2}/m^{1/2}| = |S(t)|$ on it's support. Hence, again using (\ref{eq:Sbound}), (\ref{mtrivcomt}), and (\ref{eq:dioph2}), there holds 
\begin{align*} |(\ref{int:dtmult}\te{a})| &\les \frac{1}{c\alpha}\int_{t_1}^{t_2}\|\til{J}\n{F}^{+,2}\|_{H^{\til{N}}}\frac{1}{\jap{t}^{3/2}}\|A\n{F}^{-,2}\|_{H^{\til{N} + 2 + n}}dt \\ & \les \frac{1}{c\alpha} \|\til{J}\n{F}^{+,2}\|_{L^\infty H^{\til{N}}}\|A\n{F}^{-,2}\|_{L^\infty H^{\til{N} + 2 + n}} \les \frac{\epsilon^2}{c\alpha},
\end{align*}
which is consistent. Since $|\dot{S}| \les |S(t)|^2$ and $t\jap{t}^{-2} \les \jap{t}^{-1}$, it follows that both (\ref{int:dtmult}\te{d}) and  (\ref{int:dtmult}\te{e}) can be estimated in exactly the same manner as (\ref{int:dtmult}\te{a}). We thus skip these terms and turn to (\ref{int:dtmult}\te{b}) and (\ref{int:dtmult}\te{c}), for which we have the slight variations
\begin{align*}& |(\ref{int:dtmult}\te{b})| \les \frac{\nu^{1/3}}{c\alpha}\|\til{J}\n{F}^{+,2}\|_{L^2 H^{\til{N}}}\|A\n{F}^{-,2}\|_{L^2 H^{\til{N} + 1 + n}} \les \frac{\nu^{1/3}}{c\alpha} \epsilon^2 \nu^{-1/6}\nu^{-1/6} = \frac{\epsilon^2}{c\alpha}, \\
& |(\ref{int:dtmult}\te{c})| \les \frac{1}{c\alpha}\|\sqrt{-\dot{M}{M}}\til{J}\n{F}^{+,2}\|_{L^2 H^{\til{N}}}\|\sqrt{-\dot{M}M}A\n{F}^{-,2}\|_{L^2 H^{\til{N} + 1 + n}} \les \frac{\epsilon^2}{c\alpha}.
\end{align*}
Now we consider (\ref{int:dt1}) and (\ref{int:dt2}). Expanding out (\ref{int:dt1}) gives
\begin{align} \label{eq:intexpand}
(\ref{int:dt1}) & = \frac{1}{2\alpha}\int_{t_1}^{t_2} \inp{\til{J}S \n{F}^{+,2}}{T_{2\alpha}^t \db^{-1}S\til{J}\n{F}^{-,2}}dt - \frac{1}{2\alpha}\int_{t_1}^{t_2} \inp{T_{2\alpha}^t\til{J}S \n{F}^{-,2}}{T_{2\alpha}^t \db^{-1}S\til{J}\n{F}^{-,2}}dt \\
& - \frac{\nu}{2\alpha}\int_{t_1}^{t_2} \inp{\til{J} \lap_L \n{F}^{+,2}}{T_{2\alpha}^t\db^{-1}S\til{J} \n{F}^{-,2}}dt - \mathcal{NL} \nonumber,
\end{align}
where 
\begin{equation} \label{eq:IBTNLint}
    \mathcal{NL} = \frac{1}{2\alpha}\int_{t_1}^{t_2} \inp{\til{J}(\partial_t F^{+,2})_{\mathcal{NL}}}{T_{2\alpha}^t \db^{-1} S \til{J} \n{F}^{-,2}}_{H^{\til{N}}} dt.
\end{equation}
Using that $S$ is self-adjoint, the first two terms on the right-hand side of (\ref{eq:intexpand}) are bounded as was (\ref{int:dtmult}\te{a}). For the linear term arising from the dissipation we integrate by parts:
\begin{align*}
\left|\frac{1}{2\alpha}\int_{t_1}^{t_2} \inp{\nu \lap_L \til{J} \n{F}^{+,2}}{T_{2\alpha}^t\db^{-1}\til{J}S\n{F}^{-,2}}_{H^{\til{N}}}dt\right| &= \left|\frac{1}{2\alpha}\int_{t_1}^{t_2} \inp{\nu \del_L \til{J} \n{F}^{+,2}}{T_{2\alpha}^t\db^{-1}\til{J}S\del_L\n{F}^{-,2}}_{H^{\til{N}}}dt\right| \\
& \les \frac{\nu}{c\alpha}\|\til{J} \del_L \n{F}^{+,2}\|_{L^2 H^{\til{N}}}\|A \del_L \n{F}^{-,2}\|_{L^2 H^{\til{N} + 1 + n}} \\ 
& \les \frac{\nu}{c\alpha} \epsilon^2 \nu^{-1/2} \nu^{-1/2} = \frac{\epsilon^2}{c\alpha}.
\end{align*}
The contribution to (\ref{int:dt1}) from $\mathcal{NL}$ will be considered below in Sec.~\ref{sec:intnl} along with the natural nonlinear terms that arise in the energy estimate. Regarding (\ref{int:dt2}), we observe that it can be rewritten as 
$$(\ref{int:dt2}) = \frac{1}{2\alpha} \int_{t_1}^{t_2} \inp{T_{-2\alpha}^t\db^{-1}S \til{J}\n{F}^{+,2}}{\til{J}\dt \n{F}^{-,2}}dt,$$
and hence it is essentially symmetric to (\ref{int:dt1}). It can thus be estimated in the same way. This completes the estimate of OLS.

\subsubsection{Nonlinear terms} \label{sec:intnl}
In this section we estimate the nonlinear terms in the energy estimate of Sec.~\ref{sec:Ntil} as well as the term $\mathcal{NL}$ above, which we split, corresponding to writing $(\dt F^{+,2})_{\mathcal{NL}} = \left((\dt F^{+,2})_{\mathcal{NL}} - (\dt F^{+,2})_{\mathcal{NL}}^{\neq \neq}\right) + (\dt F^{+,2})_{\mathcal{NL}}^{\neq \neq}$ in (\ref{eq:IBTNLint}), as
$$\mathcal{NL} = \mathcal{NL}^{0 \neq} + \mathcal{NL}^{\neq \neq}.$$
\par 
We first consider the interaction between the nonzero modes. For the nonlinear terms arising in the initial energy estimate we have, by (\ref{eq:int1}) and Lemma~\ref{lemhighnonzero}, 
$$ \left|\int_{t_1}^{t_2} \inp{\til{J}\n{F}^2}{\til{J}(\dt F^2)_{\mathcal{NL}}^{\neq \neq}}_{H^{\til{N}}} dt \right| \les \epsilon^3 \nu^{-1}.$$
The term $\mathcal{NL}^{\neq \neq}$ is controlled similarly. Indeed, using (\ref{eq:Sbound}), (\ref{eq:dioph2}), and $\til{N} \ll N$, we have
\begin{equation} \label{eq:absorb}
\|T_{2\alpha}^t \db^{-1} S \til{J} \n{F}^2\|_{H^{\til{N}}} \les \|A \n{F}^2\|_{H^N},
\end{equation}
and hence $|\mathcal{NL}^{\neq \neq}| \les \epsilon^3 \nu^{-1}$ by (\ref{eq:high2nonzero1}), (\ref{eq:high2nonzero2}), and Lemma~\ref{lemhighnonzero}.
\par 

Now we turn to the interactions between the zero and nonzero modes. Unlike the $(\neq,\neq)$ terms, they do not follow directly from previous calculations. This is because we used (\ref{mcom}) when controlling these interactions in Sec.~\ref{sec:Hi2nl}, and $m$ weakens more than $m^{1/2}$ near the critical times ($m^{1/2}/m$ can become size $\nu^{-1/3}$). It turns out however that due to (\ref{eq:zerovelocity}) and $\til{N} + 3 \le N$ these terms are not difficult to control, as we now demonstrate. We begin with the terms in the energy estimate written in Sec.~\ref{sec:Ntil}. By (\ref{mcom}) we obtain
\begin{align*}
|\te{NLT}(0,\neq)| &= \left|\int_{t_1}^{t_2} \int \jap{\del}^{\til{N}} \til{J} \n{F}^{2} \jap{\del}^{\til{N}} \til{J} (Z_0 \cdot \del_L \n{F}^{2})dV dt \right| \\ 
& \les \|\til{J} \n{F}^2\|_{L^2 H^{\til{N}}} \|Z_0\|_{L^\infty H^{\til{N}+1}}\|\til{J} \del_L \n{F}^2\|_{L^2 H^{\til{N}}} \\ 
& \les \epsilon^3 \nu^{-1/6} \nu^{-1/2} = \epsilon^3 \nu^{-2/3}
\end{align*}
and
\begin{align*}
|\te{NLT}(\neq,0)| &= \left|\int_{t_1}^{t_2} \int \jap{\del}^{\til{N}} \til{J} \n{F}^{2} \jap{\del}^{\til{N}} \til{J} (\n{Z} \cdot \del F_0^{2})dV dt \right| \\ 
& \les \|\til{J} \n{F}^2\|_{L^2 H^{\til{N}}} \|\n{Z}\|_{L^2 H^{\til{N}}}\|Z_0^2\|_{L^\infty H^{\til{N}+3}} \\ 
& \les \epsilon^3 \nu^{-1/6} \nu^{-1/2} = \epsilon^3 \nu^{-2/3},
\end{align*}
which both suffice. For NLS2 we have 
\begin{align*}
|\te{NLS2}(i,j,0,\neq)| &= \left|\int_{t_1}^{t_2} \int \jap{\del}^{\til{N}} \til{J} \n{F}^{2} \jap{\del}^{\til{N}} \til{J} (\di Z_0^j \dij^L \n{Z}^2)dV dt \right| \\ 
& \les \|\til{J} \n{F}^2\|_{L^2 H^{\til{N}}} \|Z_0\|_{L^\infty H^{\til{N}+2}}\|\til{J} \n{F}^2\|_{L^2 H^{\til{N}}} \\ 
& \les \epsilon^3 \nu^{-1/6} \nu^{-1/6} = \epsilon^3 \nu^{-1/3}
\end{align*}
and 
\begin{align*}
|\te{NLS2}(i,j,\neq,0)| &= \left|\int_{t_1}^{t_2} \int \jap{\del}^{\til{N}} \til{J} \n{F}^{2} \jap{\del}^{\til{N}} \til{J} (\di^L \n{Z}^j \dij Z_0^2)dV dt \right| \\ 
& \les \|\til{J} \n{F}^2\|_{L^2 H^{\til{N}}} \|A\n{F}^j\|_{L^2 H^{\til{N}}}\|Z_0^2\|_{L^\infty H^{\til{N} + 3}} \\ 
& \les \epsilon^3 \nu^{-1/6} \nu^{-1/2} = \epsilon^3 \nu^{-2/3}.
\end{align*}
Next, one can check that our methods in Sec.~\ref{sec:Hi2nl} only employed (\ref{mcom}) in the form
$$\sqrt{\frac{m(t,k,\eta',l')}{m(t,k,\eta,l)}} \les \jap{\eta - \eta'} + \jap{l - l'},$$
and hence the $(0,\neq)$ and $(\neq,0)$ interactions for NLS1 and NLP follow immediately from the estimates in Sec.~\ref{sec:Hi2nl}. For $\mathcal{NL}^{0\neq}$, we notice that by (\ref{eq:absorb}) all of the inequalities above hold with $\|\til{J} \n{F}^2\|_{L^2 H^{\til{N}}}$ replaced with $\|A \n{F}^2\|_{L^2H^N}$, which is inconsequential in the final inequality because both quantities are controlled by $\epsilon \nu^{-1/6}$.

\subsection{Estimate of $\n{F}^2$ in $H^{N'}$}
To improve estimate (\ref{eq:int2}) we use the same strategy as in Sec.~\ref{sec:Ntil}, except now the $H^{\til{N}}$ bounds in (\ref{eq:int1}) play the role of the high norm control that absorbs the loss of derivatives arising from integration by parts in time. In particular, in Sec.~\ref{sec:intlso} we observed that by using (\ref{mtrivcomt}) the gap between $A$ and $\til{J}$ could be compensated for by paying $\jap{t}^{1/2}$ and then using (\ref{eq:Sbound}). Since $J/\til{J} = \jap{t}^{1/2}$, the same structure applies to the $H^{N'}$ estimate. For example, we estimate the boundary term at $t = t_2$ analogous to that in (\ref{int:boundary}) as follows:
\begin{align*}
\frac{1}{2\alpha}\left|\inp{J\n{F}^{+,2}(t_2)}{T_{2\alpha}^{t_2} \db^{-1} S J \n{F}^{-,2}(t_2)}_{H^{N'}}\right| & \les \frac{1}{c \alpha} \|J \n{F}^{+,2}\|_{L^\infty H^{N'}}\|\jap{t}^{1/2} S \til{J}\n{F}^{-,2}\|_{L^\infty H^{N'+n}}
\\
& \les \frac{1}{c\alpha}\|J \n{F}^{+,2}\|_{L^\infty H^{N'}}\|\til{J}\n{F}^{-,2}\|_{L^\infty H^{N' + 1 + n}} \les \frac{\epsilon^2}{c \alpha}.
\end{align*}
In the last line above we have used (\ref{eq:Sbound}) and the assumption that $N' + 1 + n \le \til{N}$. The treatment of all the other terms encountered in Sec.~\ref{sec:Ntil} generalizes similarly.

\section{Low norm energy estimates} \label{sec:low}
In this section we improve (\ref{eq:low1}) and (\ref{eq:low3}). We provide the details only for (\ref{eq:low1}), as the estimate of $\n{F}^3$ follows similarly. An energy estimate gives
\begin{align*}
& \frac{1}{2}\|\til{A} \n{F}^{+,1}(t_2)\|^2_{H^{N''}} + \frac{\nu}{2}\|\til{A}\del_L \n{F}^{+,1}\|_{L^2 H^{N''}}^2 + \frac{1}{2}\|\til{m} \lam \sqrt{-\dot{M}M}\n{F}^{+,1}\|_{L^2 H^{N''}}^2 \\ 
& \le \frac{1}{2}\|\til{A}\n{F}^{+,1}(t_1)\|_{H^{N''}}^2 
-\int_{t_1}^{t_2} \inp{\til{A}\n{F}^{+,1}}{T_{2\alpha}^t\til{A} \n{F}^{-,2}}_{H^{N''}} dt 
+\int_{t_1}^{t_2} \inp{\til{A}\n{F}^{+,1}}{\til{A}\partial_{XX} \lap_L^{-1}\n{F}^{+,2}}_{H^{N''}}dt\\
& + \int_{t_1}^{t_2} \inp{\til{A}\n{F}^{+,1}}{\til{A}\partial_{XX} \lap_L^{-1}T_{2\alpha}^t\n{F}^{-,2}}_{H^{N''}}dt
+ \int_{t_1}^{t_2} \inp{\til{A} \n{F}^{+,1}}{\dX \til{A}(T_{2\alpha}^t\djj^L Z^{-,i} \di^L Z^{+,j})}_{H^{N''}}dt \\ 
& - \int_{t_1}^{t_2} \inp{\til{A} \n{F}^{+,1}}{\til{A}(T_{2\alpha}^tZ^{-}\cdot \del_L F^{+,1})}_{H^{N''}}dt 
-\int_{t_1}^{t_2}\inp{\til{A} \n{F}^{+,1}}{\til{A}(T_{2\alpha}^tF^{-}\cdot \del_L Z^{+,1})}_{H^{N''}}dt \\
& - 2\int_{t_1}^{t_2} \inp{\til{A} \n{F}^{+,1}}{\til{A}(T_{2\alpha}^t\di^L Z^{-,j} \dij^L Z^{+,1})}_{H^{N''}} dt \\ 
& = \frac{1}{2}\|\til{A}\n{F}^{+,1}(t_1)\|_{H^{N''}}^2 + \te{LU} + \te{LP1} + \te{LP2} + \te{NLP} + \te{NLT} + \te{NLS1} + \te{NLS2},
\end{align*}
where we have skipped the step of absorbing LS and $\te{L}_\lam$ into the left-hand side since it is done in the same manner as in previous estimates. Moreover, the linear stretch and linear pressure terms are dealt with exactly as in Sec.~\ref{sec:high1}, and so we skip them below.

\subsection{Nonlinear terms} \label{sec:lownl} \label{sec:NLlow}
Since $N''$ is chosen sufficiently smaller than $N'$, we see that NLT, NLS1, and NLS2 can each be controlled in the same manner as the LH interaction of the associated term in Sec.~\ref{sec:high1}. For example, for $\te{NLT}(2,\neq,\neq)$ we have the estimate 
\begin{align*}
|\te{NLT}(2,\neq,\neq)| & = \left|\int_{t_1}^{t_2} \int \jap{\del}^{N''} \til{A} \n{F}^{1} \jap{\del}^{N''} \til{A} (\n{Z}^2 \dY^L \n{F}^{1}) dV dt\right| \\ 
& \les \int_{t_1}^{t_2} \|\til{A}\n{F}^1\|_{H^{N''}}\|\jap{t}^{-1}J \n{F}^2\|_{H^{N'' + 1}}\|\del_L \n{F}^1\|_{H^{N''}} dt \\ 
& \les \nu^{-1/3} \|\til{A}\n{F}^1\|_{L^\infty H^{N''}} \|J \n{F}^2\|_{L^2 H^{N'' + 1}}\|\til{A} \del_L \n{F}^1\|_{L^2 H^{N''}} \\ 
& \les \epsilon^3 \nu^{-1/3} \nu^{-1/6} \nu^{-1/2} = \epsilon^3 \nu^{-1}.
\end{align*}
The nonlinear pressure terms also follow relatively easily due to the low regularity. Carrying out the calculations as just described completes the estimate of the nonlinear terms and moreover yields the following lemma, which will be useful in the controlling the lift-up term using integration by parts in time.

\begin{lemma}\label{lemlownonzero}
Let $j \in \{1,3\}$ and suppose that $G$ is a smooth function that satisfies 
$$\|\n{G}\|_{L^\infty H^{N''}} + \nu^{1/6}\|\n{G}\|_{L^2 H^{N''}} + \nu^{1/2}\|\del_L \n{G}\|_{L^2 H^{N''}} \les \epsilon.$$
Then, there holds
\begin{equation}
    \left|\int_{t_1}^{t_2} \inp{A_j(\dt F^j)_{\mathcal{NL}}}{\n{G}}_{H^{N''}} dt\right| \les \epsilon^3 \nu^{-1},
\end{equation}
where $A_1 = \til{A}$ and $A_3 = A$.
\end{lemma}

\subsection{Lift-up term} \label{sec:lowlu}
Now we verify that by allowing the modes of $\n{F}^{1}$ to grow indefinitely after the critical time (quantified by the use of $\til{m}$ as opposed to $m$ in the norm for $\n{F}^1$) we can treat the lift-up term with no losses. As in the previous sections, we integrate by parts in time to rewrite 

\begin{align}
-\te{LU} & = \frac{1}{2\alpha}\inp{\til{A}\n{F}^{+,1}(t_2)}{\db^{-1}T_{2\alpha}^{t_2}\til{A}\n{F}^{-,2}(t_2)}_{H^{N''}} - \frac{1}{2\alpha}\inp{\til{A}\n{F}^{+,1}(t_1)}{\db^{-1}T_{2\alpha}^{t_1}\til{A}\n{F}^{-,2}(t_1)}_{H^{N''}} \label{eq:lowboundary}\\ 
& - \frac{1}{\alpha}\int_{t_1}^{t_2} \inp{\til{A}\n{F}^{+,1}}{\frac{\dot{\til{A}}}{\til{A}} \db^{-1}T_{2\alpha}^t \til{A}\n{F}^{-,2}}_{H^{N''}} dt \label{eq:lowdtmult}\\ 
& - \frac{1}{2\alpha}\int_{t_1}^{t_2}\inp{\til{A}\dt\n{F}^{+,1}}{\db^{-1}T_{2\alpha}^t\til{A}\n{F}^{-,2}}_{H^{N''}}dt - \frac{1}{2\alpha}\int_{t_1}^{t_2}\inp{\til{A}\n{F}^{+,1}}{\db^{-1}T_{2\alpha}^t\til{A}\dt\n{F}^{-,2}}_{H^{N''}}dt \label{eq:lowdt}.
\end{align}
Due to $\til{m} \le m$, the estimate of the boundary terms is the same as in Sec.~\ref{sec:intlso}, and so we skip it and move on to (\ref{eq:lowdtmult}). We split (\ref{eq:lowdtmult}) into 
$$(\ref{eq:lowdtmult}) = (\ref{eq:lowdtmult}\te{a}) + (\ref{eq:lowdtmult}\te{b}) + (\ref{eq:lowdtmult}\te{c}),$$
corresponding to 
$$\frac{\dot{\til{A}}}{\til{A}} = \frac{\dot{M}}{M} + \delta \nu^{1/3} + \frac{\dot{\til{m}}}{\til{m}}.$$
By (\ref{eq:dioph2}), (\ref{mtildecay}), and $\til{m} \le m$ we have 
\begin{align*}
| (\ref{eq:lowdtmult}\te{a})| & \les \frac{1}{c\alpha}\int_{t_1}^{t_2} \|\til{A}\n{F}^{+,1}\|_{H^{N''}}\frac{1}{\jap{t}}\|\sqrt{-\dot{M}M}J \n{F}^{-,2}\|_{H^{N'' + n + 2}} dt \\ 
& \les \frac{1}{c\alpha}\|\til{A}\n{F}^{+,1}\|_{L^\infty H^{N''}}\|\sqrt{-\dot{M}M}J \n{F}^{-,2}\|_{L^2 H^{N'}} \les C_0\frac{\epsilon^2}{c\alpha},
\end{align*}
which suffices for $\alpha \gg C_0/c$. The estimate of (\ref{eq:lowdtmult}\te{c}) follows from similar techniques and the fact that $|\dot{\til{m}}/{\til{m}}| \les \jap{t}^{-1}|k,\eta,l|$ on its support:
\begin{align*}
| (\ref{eq:lowdtmult}\te{c})| & \les \frac{1}{c\alpha}\int_{t_1}^{t_2} \|\til{A}\n{F}^{+,1}\|_{H^{N''}}\frac{1}{\jap{t}^2}\|J \n{F}^{-,2}\|_{H^{N'' + n + 3}} dt \\ 
& \les \frac{1}{c\alpha}\|\til{A}\n{F}^{+,1}\|_{L^\infty H^{N''}}\|J \n{F}^{-,2}\|_{L^\infty H^{N'}} \les C_0\frac{\epsilon^2}{c\alpha}.
\end{align*}
Controlling (\ref{eq:lowdtmult}\te{b}) is the same as the analogous term in Sec.~\ref{sec:intlso}, and so we omit the details. This completes the estimate of (\ref{eq:lowdtmult}). \par 
Next we consider (\ref{eq:lowdt}). The linear terms do not require any methods beyond what we have employed thus far, and so we will only sketch how to deal with them. The dissipation terms are easily estimated using integration by parts as in Sec.~\ref{sec:intlso}. Next, both $\dt \n{F}^{+,1}$ and $\dt \n{F}^{-,2}$ contain a LS (or OLS) term. Each of these terms carries a factor of $S(t)$ (recall the notation defined in Sec.~\ref{sec:intlso}), and so by using (\ref{eq:Sbound}) we bound these terms as we did (\ref{eq:lowdtmult}\te{c}). There are also linear contributions from LP1 and LP2 in $\dt\n{F}^{+,1}$. Using that $|\partial_{XX}\lap_{L}^{-1}| \les -\dot{M}M$, these are terms bounded like (\ref{eq:lowdtmult}a) above. Lastly, there is a linear term that arises from the lift-up term in $\dt\n{F}^{+,1}$, but this term vanishes like the analogous term did in Sec.~\ref{sec:luhi}.

Now we turn to the nonlinear terms created in (\ref{eq:lowdt}), which we write as (dropping the irrelevant minus signs and factors of $\alpha^{-1}$) 
\begin{align*}
    \int_{t_1}^{t_2}\inp{\til{A}(\dt F^{+,1})_{\mathcal{NL}}}{\db^{-1}T_{2\alpha}^t\til{A}\n{F}^{-,2}}_{H^{N''}}dt & + \int_{t_1}^{t_2}\inp{\til{A}\n{F}^{+,1}}{\db^{-1}T_{2\alpha}^t\til{A}(\dt F^{-,2})_{\mathcal{NL}}}_{H^{N''}}dt \\ 
    & = \mathcal{NL}_1 + \mathcal{NL}_2.
\end{align*}
Since $\|\db^{-1}T_{2\alpha}^t \til{A}\n{F}^2\|_{H^{N''}} \les \|A \n{F}^2\|_{H^N}$
it follows by (\ref{eq:high2nonzero1}), (\ref{eq:high2nonzero2}), and Lemma~\ref{lemlownonzero} that $|\mathcal{NL}_1| \les \epsilon^3 \nu^{-1}. $
The term $\mathcal{NL}_2$ is less immediate since the loss of regularity caused by $\db^{-1}$ implies that we must appeal to bootstrap hypotheses (\ref{eq:high1}) and (\ref{eq:high3}). We first notice that by an integration by parts, (\ref{eq:dioph2}), and (\ref{mtildecay}) there holds
\begin{align*}
|\mathcal{NL}_2| & \les \int_{t_1}^{t_2} \|\del_L \til{A} \n{F}^1\|_{H^{N''}}\jap{t}^{-1}\|\lam \jap{t}^{-1}\del_L (Z\cdot \del_L Z^2)_{\neq}\|_{H^{N'' + n + 4}}dt \\ 
& + \int_{t_1}^{t_2} \|\til{A} \n{F}^1\|_{H^{N''}}\jap{t}^{-1}\|\lam \jap{t}^{-1}\dY^L (\djj^L Z^i\di^L Z^j)_{\neq}\|_{H^{N'' + n + 4}}dt \\ 
& \les \epsilon \nu^{-1/2} \|\lam \jap{t}^{-1}\del_L (Z\cdot \del_L Z^2)_{\neq}\|_{L^\infty H^{N'' + n + 4}} \\ 
& + \epsilon \nu^{-1/6}\|\lam \jap{t}^{-1}\dY^L(\djj^L Z^i \di^L Z^j)_{\neq}\|_{L^\infty H^{N'' + n + 4}},
\end{align*}
and hence to complete the desired estimate under the assumptions of Theorem~\ref{thrm:main} it suffices to prove that 
\begin{align} 
\|\lam \jap{t}^{-1}\del_L (Z\cdot \del_L Z^2)_{\neq}\|_{L^\infty H^{N'' + n + 4}} & \les \epsilon^2 \nu^{-1/2} \label{eq1},\\ 
\|\lam \jap{t}^{-1}\dY^L(\djj Z^i \di Z^j)_{\neq}\|_{L^\infty H^{N'' + n + 4}} & \les \epsilon^2 \nu^{-5/6}. \label{eq2}
\end{align}
To prove (\ref{eq1}) we use $N'' + n + 5 \le N' \le N$ to obtain 
\begin{align*}
\|\lam \jap{t}^{-1}\del_L (Z\cdot \del_L Z^2)_{\neq}\|_{H^{N'' + n + 4}} & \les \|\lam (Z \cdot \del_L Z^2)_{\neq}\|_{H^{N'' + n + 5}} \\ 
& \les \|\lam \n{Z}\|_{H^N}  \|\del_L Z^2\|_{H^{N'}} + \|Z_0\|_{H^N}  \|\lam \del_L \n{Z}^2\|_{H^{N'}} \\ 
& \les (\|\lam \n{Z}\|_{H^N} + \|Z_0\|_{H^N})(\|J\n{F}^2\|_{H^{N'}} + \|Z_0\|_{H^N}) \\
& \les \epsilon^2 \nu^{-1/3},
\end{align*}
which suffices. For (\ref{eq2}) we have 
\begin{align*}
\|\lam \jap{t}^{-1}\dY^L(\djj^L Z^i \di^L Z^j)_{\neq}\|_{H^{N'' + n + 4}} & \les \|\lam (\djj^L Z^i \di^L Z^j)_{\neq}\|_{H^{N'' + n  + 5}} \\ 
& \les \epsilon^2 \nu^{-2/3},
\end{align*}
where the last line follows by, as in previous estimates, using (\ref{mlapl}), (\ref{mtrivcom}), and considering separately the cases $i,j \in \{1,3\}$, $i = 2$ and $j \neq 2$, and $i = j = 2$. 

\section{Zero mode velocity estimates} \label{sec:zero}
In this section we improve (\ref{eq:zerovelocity}). For any $r \in \{1,2,3\}$, an energy estimate gives 
\begin{align*}
& \frac{1}{2}\|Z_0^{+,r}(t_2)\|^2_{H^N} + \nu\|\del Z^{+,r}_0\|_{L^2 H^N}^2 = \frac{1}{2}\|Z_0^{+,r}(t_1)\|_{H^N}^2 - \ind_{r = 1}\int_{t_1}^{t_2} \int \DN Z^{+,1}_0 \DN T_{2\alpha}^t Z^{-,2}_0 dV dt \\ 
& -\int_{t_1}^{t_2} \int \DN Z^{+,r}_0 \DN \left(T_{2\alpha}^tZ^{-} \cdot \del_L Z^{+,r}dVdt\right) \\ 
& + \ind_{r \neq 1}\int_{t_1}^{t_2} \int \DN Z^{+,r}_0 \DN \partial_r\lap^{-1}(T_{2\alpha}^t\djj Z^{-,i} \di Z^{+,j})_0dV dt \\ 
& = \frac{1}{2}\|Z_0^+(t_1)\|_{H^N}^2 + \te{LU} + \te{NLT} + \te{NLP}.
\end{align*}
The lift-up term can be dealt with using integration by parts in time as in Sec.~\ref{sec:luhi}. We skip it and turn to the nonlinear terms, beginning with the transport term
$$\te{NLT}(j) = \int_{t_1}^{t_2} \int \DN Z_0^r \DN (Z^j\djj Z^r)_0 dV dt.$$
First we treat the interaction between the nonzero modes. When $j = 1,3$ we have 
\begin{align*}
|\te{NLT}(j,\neq,\neq)| & \les \|Z_0^r\|_{L^\infty H^N}\|\til{A}\n{F}^j\|_{L^2 H^N} \|\til{A} \n{F}^r\|_{L^2 H^N} \\ 
& \les \epsilon^3 \nu^{-1/2}\nu^{-1/2} = \epsilon^3 \nu^{-1},
\end{align*}
while for $j = 2$ there holds
\begin{align*}
|\te{NLT}(2,\neq,\neq)| & \les \|Z_0^r\|_{L^\infty H^N} \|A \n{F}^2\|_{L^2 H^N}\nu^{-1/3}\|\til{A}\n{F}^r\|_{L^2 H^N} \\ 
& \les \epsilon^3 \nu^{-1/6} \nu^{-1/3} \nu^{-1/2} = \epsilon^3 \nu^{-1}.
\end{align*}
For the interaction between the zero modes, we use the divergence free condition to obtain
\begin{align*}
|\te{NLT}(2,0,0)| & \les \|Z_0^r\|_{L^\infty H^N} \|\del Z_0^2\|_{L^2 H^N} \|\del Z_0^r\|_{L^2 H^N} \\ 
& \les \epsilon^3 \nu^{-1/2} \nu^{-1/2} = \epsilon^3 \nu^{-1}.
\end{align*}
The NLT(3,0,0) term is bounded similarly by employing the method used to treat the associated term in Sec.~\ref{sec:nonlinearhizero}. We omit the details. Turning now to the pressure, we observe that by using incompressibility and integration by parts it can be written as
$$\te{NLP}(i,j) = -\int_{t_1}^{t_2} \int \DN \partial_r Z_0^r \jap{\del}^N \dij \lap^{-1} (Z^i Z^j)_0 dV dt, $$
where the term is nonzero only for $i,j \in \{2,3\}$. For the interaction between the nonzero modes we use that $\dij\lap^{-1}$ is bounded on $H^N$ along with a paraproduct decomposition to obtain 
\begin{align*}
|\te{NLP}(i,j,\neq,\neq)| & \les \|\del Z_0^r\|_{L^2 H^N}\|\n{Z}^j\|_{L^2 H^{N''}}\|\n{Z}^{i}\|_{L^\infty H^N} + \te{ symmetric term} \\ 
& \les \epsilon^3 \nu^{-1/2} \nu^{-1/6} \nu^{-1/3} = \epsilon^3 \nu^{-1},
\end{align*}
which is consistent. For $\te{NLP}(i,j,0,0)$ we use that $i,j \in \{2,3\}$ along with incompressibility implies that at least one of $Z_0^i$ or $Z_0^j$ has a nonzero $Z$-frequency. Thus, there holds 
\begin{align*}
|\te{NLP}(i,j,0,0)| &\les \|\del Z_0^r\|_{L^2 H^N} \|\del Z_0^i \|_{L^2 H^N} \|Z_0^j\|_{L^\infty H^N} + \te{ symmetric term} \\ 
& \les \epsilon^3 \nu^{-1/2} \nu^{-1/2} = \epsilon^3 \nu^{-1}.
\end{align*}

\section*{Acknowledgements}
The author would like to thank his advisor, Jacob Bedrossian, for suggesting an MHD stability problem and providing guidance throughout. This work was partially supported by Jacob Bedrossian's NSF CAREER grant DMS-1552826 and NSF RNMS \#1107444 (Ki-Net).

\bibliographystyle{abbrv} 

\end{document}